\documentclass[12pt]{amsart}
\usepackage{amsmath,amssymb,euscript,mathrsfs, eufrak,tikz}
\usepackage[all]{xy}
\usepackage{enumerate}

\pushQED{\qed}

\newcommand{\define}{\textbf}

\renewcommand{\setminus}{\smallsetminus}
\renewcommand{\phi}{\varphi}

\renewcommand{\tilde}{\widetilde}
\renewcommand{\hat}{\widehat}
\renewcommand{\bar}{\overline}
\renewcommand{\L}{\mathbb{L}}

\newcommand{\D}{\mathbb{D}}

\newcommand{\C}{\mathbb{C}}
\newcommand{\Q}{\mathbb{Q}}
\newcommand{\R}{\mathbb{R}}
\newcommand{\N}{\mathbb{N}}
\newcommand{\Z}{\mathbb{Z}}
\renewcommand{\P}{\mathbb{P}}
\newcommand{\A}{\mathbb{A}}

\newcommand{\K}{\mathbb{K}}

\newcommand{\cS}{\mathcal{S}}

\newcommand{\lc}{l^*}

\newcommand{\Gr}{Gr}

\newcommand{\bSigma}{\mbox{\boldmath$\Sigma$}}

\DeclareMathOperator{\Hom}{Hom}
\DeclareMathOperator{\Spec}{Spec}
\DeclareMathOperator{\ord}{ord}

\DeclareMathOperator{\bx}{Box}

\DeclareMathOperator{\pt}{pt}

\DeclareMathOperator{\lk}{lk}

\DeclareMathOperator{\Int}{Int}

\DeclareMathOperator{\Vol}{Vol}
\DeclareMathOperator{\Trop}{Trop}

\DeclareMathOperator{\prim}{prim}

\DeclareMathOperator{\Var}{Var}

\DeclareMathOperator{\init}{in}

\DeclareMathOperator{\gen}{gen}

\DeclareMathOperator{\inter}{int}

\DeclareMathOperator{\Lef}{Lef}
\DeclareMathOperator{\NP}{NP}

\DeclareMathOperator{\pr}{pr}
\DeclareMathOperator{\orb}{orb}
\DeclareMathOperator{\Sp}{sp}

\newtheorem{theorem}{Theorem}[section]
\newtheorem{lemma}[theorem]{Lemma}

\newtheorem{corollary}[theorem]{Corollary}

\newtheorem*{ntheorem}{Theorem}

\theoremstyle{definition}
\newtheorem{definition}[theorem]{Definition}
\newtheorem{remark}[theorem]{Remark}
\newtheorem{example}[theorem]{Example}

\newcommand{\excise}[1]{}

\begin{document}

\title[]{
Formulas for monodromy
}

\author{Alan Stapledon}
\email{astapldn@gmail.com}

\date{}
\thanks{}

\begin{abstract}

Given a family $X$  of complex varieties degenerating over a punctured disc, one is interested in computing related invariants called the motivic nearby fiber and the refined limit mixed Hodge numbers,
both of which contain information about the induced action of monodromy on the cohomology of a fiber of $X$. Our first main result is that the motivic nearby fiber of $X$ can be computed
by first stratifying $X$ into locally closed subvarieties that are  non-degenerate in the sense of Tevelev, and then applying an explicit formula on each piece of the stratification that involves tropical geometry.
Our second main result is an explicit combinatorial formula for the refined limit mixed Hodge numbers in the case when $X$ is  a  family of non-degenerate hypersurfaces.
As an application, given a complex polynomial, then, under appropriate conditions, we give a combinatorial formula for the Jordan block structure of the action of monodromy on the cohomology of the Milnor fiber,  generalizing a famous formula of Varchenko for the associated eigenvalues. In addition, we give a formula for the Jordan block structure of the action of monodromy at infinity.

\end{abstract}

\maketitle
\tableofcontents

\section{Introduction}

Before introducing the main results of this paper, we first present an application.
Let $f(x_1, \ldots, x_n) \in \C[x_1, \ldots, x_n]$ be a complex polynomial with no constant term, and consider the induced map $f: \C^n \rightarrow \C$.
One would like to understand the singularity of $f^{-1}(0) \subseteq \C^n$ at the origin (we assume this is an isolated singularity). A classical invariant used to distinguish the singularities arising
from different polynomials is the action of monodromy on the cohomology of the associated \define{Milnor fiber}. In simple terms, this means associating to $f$ a finite, square complex matrix $M_f$ that is well-defined up to change of basis. 
We assume that $f$ is `general' in the sense that it is non-degenerate with respect to its \define{Newton polyhedron} $\Gamma_+(f)$, and also that  it is convenient (see Section~\ref{s:Milnor}).
The most famous result was proved by Varchenko in 1976 \cite{VarZeta}, who gave an explicit combinatorial formula for the eigenvalues (with multiplicity) of $M_f$, involving alternating signs of
volumes of polytopes associated to 
$\Gamma_+(f)$. Unfortunately, since $M_f$ is rarely diagonalizable, it is not, in general, determined by its eigenvalues. We are left with the problem of
giving combinatorial formulas for the numbers $J_{k,\alpha}$  of Jordan blocks of $M_f$ of size $k$ with eigenvalue $\alpha$.
Using \define{tropical geometry}, together with 
\define{weighted Ehrhart theory}, 
which involves the enumerative combinatorics of lattice points in dilates of polytopes, we settle this question:

\begin{ntheorem}[Corollary~\ref{c:jordanMilnor}]
Let  $f \in \C[x_1,\ldots, x_n]$  be a complex polynomial such that $f^{-1}(0)$ admits an isolated singularity at the origin. Assume further that $f$ is convenient and non-degenerate with respect to its Newton polyhedron
$\Gamma_+(f)$.
Then there are explicit, non-negative combinatorial formulas for the numbers $J_{k,\alpha}$, involving the weighted Ehrhart theory of polytopes associated to $\Gamma_+(f)$.
\end{ntheorem}

We refer the reader to Corollary~\ref{c:jordanMilnor} for the specific formulas.
By a `non-negative' formula, we mean that each $J_{k,\alpha}$ is expressed as a sum of non-negative integers. By summing over $k$, we obtain a non-negative formula for the multiplicities of the eigenvalues of
$M_f$, which is equivalent to Varchenko's formula (Remark~\ref{r:var}).
Using different techniques,
special cases of formulas for  the $J_{k,\alpha}$ have previously been proved by
Matsui and Takeuchi \cite{MTMilnor} and van Doorn and Steenbrink \cite{DSSupplement}, and an algorithm to compute the numbers $J_{k,\alpha}$  involving simplicial resolutions of toric varieties was proved by
 Matsui and
Takeuchi in  \cite{MTMilnor}, extending work of Danilov  \cite{DanNewton} and Tanab\'e \cite{TanCombinatorial}.
We note that  algorithms in a  more general setting are given by Schulze in \cite{SchAlgorithms}.

 The theorem above is a corollary of more general results, 
 that we will outline in the remainder of the introduction.
Given a family of varieties $f:X\rightarrow\D^*$ over a complex disc, Denef and Loeser  \cite{DLGeometry} associate an invariant $\psi_X$ called the \define{motivic nearby fiber}, which contains information about the
variation of Hodge structures associated to the degeneration. In particular, if we fix a non-zero fiber $X_{\gen} := f^{-1}(t)$ for some sufficiently small $t \in \D^*$, then $\psi_X$ encodes information on the induced monodromy map $T: H^m_c(X_{\gen}) \rightarrow  H^m_c(X_{\gen})$, which is a linear operator on the complex cohomology with compact supports of $X_{\gen}$ and is quasi-unipotent i.e. $T = T_sT_u$, where
$T_s$ and $T_u$ commute, $T_s$ is semi-simple and has finite order, and $T_u$ is unipotent.
In \cite{Geometry}, Eric Katz and the author showed that one can use tropical geometry to concretely compute
a specialization of $\psi_X$ that is invariant  under base change, and hence encodes information about $T_u$ but not $T_s$.
Using new combinatorics developed in \cite{Combinatorics}, in the case of families of sch\"{o}n complex hypersurfaces of tori, explicit combinatorial formulas were then deduced for
 the \define{refined limit mixed Hodge numbers}, which, in particular, determine
 the Jordan block structure of
the action of $T_u$ on the graded pieces (with respect to the Deligne weight filtration) of  $H^m_c(X_{\gen})$. Here \define{sch\"on} is a `generic' condition introduced by
Tevelev \cite{TevComp}, generalizing the notion of  non-degeneracy of a hypersurface of a complex torus \cite{KhoNewton}.
Equivalent formulas were also given for the action of $T_u$ on the intersection cohomology groups of a family of sch\"on hypersurfaces of a projective toric variety.

The  goal of this note is to extend the above results to the full generality of the motivic nearby fiber $\psi_X$ and monodromy operator $T$. Our first main result is to show that
the motivic nearby fiber may be computed using tropical geometry (see Section~\ref{s:trop}). A key point is that  since there is a range of available software implementing the main algorithms in tropical geometry \cite{GFan,tropicallib}, our main formula
Theorem~\ref{t:comp} can be computed in practice.

Weighted Ehrhart theory was introduced by the author in  \cite{Weighted} both in order to  extend and reprove many results in
Ehrhart theory, the study of enumeration of lattice points in polytopes, and in order to give explicit computations of  motivic integrals on toric stacks \cite{Stacks}. Roughly speaking, one associates
to every lattice point $v$ a `weight' $w(v)$ in $\Q/\Z$, and then attempts to enumerate lattice points in polytopes keeping track of the associated weights.
Extending work in \cite{Combinatorics}, in Section~\ref{s:mixed}, we associate new 
Ehrhart-theoretic invariants to a pair $(P,\nu)$, where $P$ is a lattice polytope and $\nu$ is the convex graph of an integral height function on $P$.
In particular, we introduce the \define{weighted refined limit mixed $h^*$-polynomial} $h^*(P,\nu;u,v,w) \in \Z[\Q/\Z][u,v,w]$. Roughly speaking, our second main result states that if $X^\circ$ is a family of
sch\"on complex hypersurfaces of tori, then the associated \define{equivariant refined limit mixed Hodge numbers} are precisely
the coefficients of  $h^*(P,\nu;u,v,w)$ (Theorem~\ref{t:mainhyper}, Corollary~\ref{c:explicit}), where $(P,\nu)$ is the Newton polytope and associated convex graph associated to $X^\circ$.
In particular, this gives combinatorial formulas for the Jordan  block structure of
the action of $T$ on the graded pieces (with respect to the Deligne weight filtration) of  $H^m_c(X^\circ_{\gen})$.
Equivalent formulas are also given for the equivariant refined limit mixed Hodge numbers associated to the intersection cohomology groups of a  family of sch\"on hypersurfaces of a projective toric variety (Theorem~\ref{t:mainintersection}, Corollary~\ref{c:explicit}). In Example~\ref{e:orbifold}, we express these invariants in a special case in terms of dimensions of orbifold cohomology  \cite{CROrbifold, CRNew} of a
toric stack \cite{BCSOrbifold}. It would be interesting to have an explanation of this fact involving `mirror symmetry'.

As a special case of the above result, in Theorem~\ref{t:DKformula}, we give a formula for the equivariant Hodge-Deligne polynomial of a sch\"on complex hypersurface of a torus invariant under the action of an element of the torus of finite order.
An algorithm to determine this polynomial had previously been given by Matsui and Takeuchi in \cite{MTMonodromy}, generalizing an algorithm of Danilov and Khovanski{\u\i} in the non-equivariant setting \cite{DKAlgorithm}. Moreover, the above formula reduces to a formula of  Borisov and Mavlyutov \cite{BMString} in the non-equivariant setting, which itself is a simplification of a formula of Batyrev and Borisov
\cite{BBMirror}.

We next present some applications to the monodromy of complex polynomials. Let $f(x_1, \ldots, x_n) \in \C[x_1, \ldots, x_n]$ be a complex polynomial. It is well-known that there exists a minimal finite subset $B_f \subseteq \C$,
such
that $f: \C^n \rightarrow \C$ is a locally trivial fibration away from $B_f$. Then monodromy defines an action of the fundamental group $\pi_1(\C \smallsetminus B_f)$  on the cohomology of a fixed generic fiber of $f$. 
On the one hand, we are interested in the monodromy action induced by moving anti-clockwise around a small loop around an element of $b \in B_f$. After translating $f$ by a constant, we may assume that $b = 0$, and then
the resulting action is called \define{monodromy at 0}.
On the other hand, choose $R > 0$ sufficiently large such that $B_f$ is strictly contained in $\{ z \in \C \mid |z| = R \}$. Then the  monodromy action on cohomology induced by moving clockwise around such a loop is called the \define{monodromy at infinity} of $f$, and
a fundamental result of
Dimca and N\'emethi \cite{DNMonodromy} states that monodromy at infinity essentially determines the monodromy action of $\pi_1(\C \smallsetminus B_f)$.
We assume that $f$ is convenient in the sense that its Newton polytope has non-zero intersection with each  ray through a coordinate vector \cite{Kou}.
Then under certain sch\"onness (also called `non-degeneracy') conditions, we give explicit combinatorial formulas that essentially completely describe
both monodromy at $0$ (Example~\ref{e:0}) and
monodromy at infinity (Example~\ref{e:infty}). In particular, we deduce `non-negative' combinatorial formulas for both the equivariant limit mixed Hodge numbers associated to monodromy at infinity
(Corollary~\ref{c:formula}), and the Jordan block structure of monodromy at infinity (Corollary~\ref{c:jordaninfty}). Algorithms to compute the latter invariants were given by Matsui and
Takeuchi in  \cite{MTMonodromy}. Also, these formulas specialize to known formulas for the spectrum at infinity of $f$ \cite[Theorem~5.11]{MTMonodromy} and zeta function at infinity of $f$
\cite{LSZeta}. We note that some of the above results may be extended without the assumption that $f$ is convenient (see Remark~\ref{r:notconvenient}).
We also present analogous `local' results for the Milnor fiber of a polynomial $f$ with an isolated singularity at the origin, including combinatorial formulas for both the corresponding the equivariant limit mixed Hodge numbers,
and, as outlined at the beginning of this introduction, the Jordan block structure of the action of monodromy on the cohomology of the Milnor fiber.


Finally, we mention some possible generalizations of the above results that we expect will follow by similar methods.
Firstly,
using the `Cayley trick' of Danilov and Khovanski{\u\i} \cite[Section~6]{DKAlgorithm}, one may extend our results on hypersurfaces to the case of complete intersections.
In particular, in the case of the monodromy at infinity and monodromy of the cohomology of the Milnor fiber, one may obtain explicit formulas that  extend algorithms given by Esterov and Takeuchi in \cite{ETMotivic}.
Secondly, in the case when $f$ is not convenient, one may use the results above to obtain formulas extending algorithms of  Takeuchi and Tibar in \cite{TTMonodromies}. Lastly, analogous formulas to those in Section~\ref{s:Milnor}
 for families of hypersurfaces of an affine toric variety (rather than simply affine space) may
be obtained using the setup and results of Steenbrink in  \cite{SteMotivic}.

 \subsection{Organization of the paper}

     This paper is structured as follows. In Section~\ref{s:degenerations}, we review the necessary geometric background. In Section~\ref{s:trop}, we show how to compute the motivic nearby fiber using tropical
     geometry. In Section~\ref{s:weighted}, we review the necessary combinatorial background before introducing our main combinatorial invariants. In Section~\ref{s:hyper}, we prove our combinatorial formulas
     for geometric invariants associated to degenerations of hypersurfaces. Finally, in Section~\ref{s:sing}, we apply our results to deduce formulas for invariants associated with the monodromy of complex polynomials.


\medskip
\noindent
{\it Notation and conventions.}
All cohomology has complex coefficients.
If $N$ is a lattice, then $N_\R = N \otimes_\R \R$. We identify the group  $\Q/\Z$ with the group $\mathbb{S}^1_\Q$ of rational points on the circle $\{ z \in \C \mid |z| = 1 \}$, sending $[k] \in \Q/\Z$ to $\alpha = e^{2 \pi \sqrt{-1}k} \in \mathbb{S}^1_\Q$. We fix $\K = \C(t)$.
A variety $X$ over $\K$ is naturally interpreted as a complex variety with a morphism $f:X\rightarrow \D^*$, for some sufficiently small punctured complex disc $\D^*$ around the origin,
and  $X_{\gen} := f^{-1}(t)$ denotes a fixed fiber for some choice of $t \in \D^*$.


{\it Acknowledgements.} 
The author would like to thank Anthony Henderson for some useful comments.

\section{Geometry of degenerations}\label{s:degenerations}

In this section we briefly recall the notion of the motivic nearby fiber and its associated geometry. We refer the reader to  \cite{BitMotivic}, \cite{Peters}  and \cite{PSMixed} for details.
A similar exposition  is given in the non-equivariant setting in \cite[Section~3]{Geometry}.


\subsection{The equivariant Grothendieck ring}

We follow the treatment in \cite{BitMotivic} below.
Fix a field $k$.
The \define{Grothendieck ring} $K_0(\Var_k)$ of algebraic varieties over $k$ is the free abelian group generated by isomorphism classes $[V]$ of varieties $V$ over $k$, modulo  the relation
\begin{equation}\label{e:cut}
[V] = [U] + [V \setminus U],
\end{equation}
whenever $U$ is an open subvariety of $V$.  Multiplication is given by tensor product of varieties.

 Suppose that $G$ is a finite group. An (algebraic) action of $G$ on a variety over $k$ is \define{good} if every orbit lies in an open affine subvariety.
 Then we may similarly define $\tilde{K}_0^{G}(\Var_k)$ to be the
 free abelian group generated by isomorphism classes $[V]$ of varieties $V$ over $k$ together with a good action of $G$, modulo  the relation
$[V] = [U] + [V \setminus U]$,
whenever $U$ is a $G$-invariant open subvariety of $V$.  The \define{equivariant Grothendieck ring}  $K_0^{G}(\Var_k)$ is the quotient of
 $\tilde{K}_0^{G}(\Var_k)$ by the relation $[\P(V) \circlearrowleft G] = [\P^n \times (X \circlearrowleft G)]$, where $V \rightarrow X$ is a vector bundle of rank $n + 1$ with a good $G$-action that
 is linear over a good $G$-action on $X$. Here $\P(V)$ denotes the projectivization of $V$ with induced $G$-action, and $\P^n$ admits a trivial $G$-action.
 A group homomorphism $G \rightarrow H$ induces a ring homomorphism $K_0^{H}(\Var_k) \rightarrow K_0^{G}(\Var_k)$ by restriction. In particular,  $K_0^{G}(\Var_k)$ is a module over
 $K_0(\Var_k)$, and
 forgetting the action of $G$ gives a ring homomorphism $K_0^{G}(\Var_k) \rightarrow K_0(\Var_k)$.
  We will follow the convention that $\L := [\A^1]$ is
  the class of the affine line with trivial $G$-action.
A \define{motivic invariant} is a ring homomorphism $K_0^{G}(\Var_k) \rightarrow R$, for some ring $R$.

\begin{remark}\label{r:representation}
Suppose that a finite group $G$ acts linearly on a $k$-vector space $V$ of dimension $n + 1$ i.e.  $V$ is a representation of $G$. Then taking $X = \pt$ in the condition above, we see that
$[\P(V)] = [\P^n] = \L^n + \L^{n - 1} + \cdots + 1 \in K_0^{G}(\Var_k)$.
\end{remark}

Now assume that $k = \C$ and identify $\mu_n \cong \Z/n\Z$ with the group of $n^{\textrm{th}}$ roots of unity in $\C$
generated by $\exp(2 \pi \sqrt{-1} i/n)$. The epimorphisms $\mu_{nd} \rightarrow \mu_n, \zeta \mapsto \zeta^d$
give rise to a projective limit $\hat{\mu} = \varprojlim \mu_n$. 
If $g$ is an element of a group $G$ of finite order, then we have a well-defined
group homomorphism $\hat{\mu} \rightarrow \mu_n \rightarrow G$ sending $\exp(2 \pi \sqrt{-1} i/n) \mapsto g$ whenever $g^n = 1$. In this way, 
group actions of elements of finite order give rise
to 
group actions of $\hat{\mu}$.

A \emph{good $\hat{\mu}$-action} on a complex variety is an action of $\hat{\mu}$ induced by a good $\mu_n$-action for some $n > 0$.
The equivariant Grothendieck ring $K_0^{\hat{\mu} }(\Var_\C)$ is defined similarly as above,  and
 coincides with the direct limit of the restriction maps $K_0^{\mu_{n}}(\Var_\C) \rightarrow K_0^{\mu_{nd}}(\Var_\C)$ induced by
 the epimorphisms $\mu_{nd} \rightarrow \mu_n, \zeta \mapsto \zeta^d$.

\begin{example}\label{e:torus}
Let $v = (v_1, \ldots, v_n) \in \Q^n$.
 Then multiplication by $(e^{2 \pi \sqrt{-1}v_1}, \ldots, e^{2 \pi \sqrt{-1}v_n})$ gives a good  $\hat{\mu}$-action
on $(\C^*)^n$. This action extends to $\P^n$, which admits a natural stratification into tori each invariant under $\hat{\mu}$, and it follows from
Remark~\ref{r:representation} and induction on dimension that $[(\C^*)^n \circlearrowleft \hat{\mu}] = (\L - 1)^n \in K_0^{\hat{\mu} }(\Var_\C)$.
\end{example}

\subsection{The motivic nearby fiber}\label{s:mnf}
Fix $\K = \C(t)$.
Then the motivic nearby fiber is a ring homomorphism
\[
\psi : K_0(\Var_\K) \rightarrow K_0^{\hat{\mu} }(\Var_\C),
\]
\[
[X] \mapsto \psi_X.
\]
A result of Bittner \cite{BitUniv} implies that $K_0(\Var_\K)$ is generated by the classes of smooth, proper varieties. In particular, the description below determines $\psi$.

We follow the description of the motivic nearby fiber in  \cite{SteMotivic}.
A variety $X$ over $\K$ is naturally interpreted as a complex variety with a morphism $f:X\rightarrow \C^* \setminus \{ b_1, \ldots , b_r \}$, for some points $b_1,\ldots, b_r$.
Assume that $X$ is smooth and extend $X$ to a variety $X'$ with flat morphism $f: X' \rightarrow \A^1 \setminus \{ b_1, \ldots , b_r \}$. After resolving singularities, we may assume
that $X'$ is smooth and
that the central fiber $f^{-1}(0) = \sum_i m_i D_i$ is a simple normal crossings divisor with irreducible components $D_1,\ldots,D_r$, with multiplicities $m_1,\ldots, m_r$ respectively.
For each non-empty subset  $I \subseteq \{ 1, \ldots ,  r \}$, let  $D_I^\circ = \cap_{i \in I} D_i \setminus \cup_{j \notin I} D_j$ and  $m_I= \gcd(m_i \mid i \in I)$.
Restrict $f$ to a small complex disc $\D$ about the origin
 and let $m$ be a common multiple of $m_1,\ldots, m_r$.
 Pullback $X'$ via the map
$\D \rightarrow \D$, $t \mapsto t^m$ and normalize to obtain a commutative diagram
\[\xymatrix{
\tilde{X'}  \ar[d]^{\tilde{f}} \ar[r]^{\rho}  & X'  \ar[d]^f \\
\D \ar[r]^{t \mapsto t^m}  & \D. \\
}\]
The action of $\mu_m$ on $\D$, with $\zeta \in \mu_m$ acting via $t \mapsto \zeta t$, extends to an action on $\tilde{X'}$, and hence a good $\hat{\mu}$-action on the central fiber.
Let $\tilde{D_I^\circ}$ be proper transform of $D_I^\circ$ via $\rho$. Then  $\rho:  \tilde{D_I^\circ} \rightarrow D_I^\circ$ is a $\mu_{m_I}$-covering.
 The motivic nearby fiber $\psi_X \in K_0^{\hat{\mu} }(\Var_\C)$ of $X$ is given by
\begin{equation}\label{e:dmotivic}
\psi_X = \sum_{\emptyset \ne I \subseteq \{ 1, \ldots ,  r \}} [\tilde{D}_I^\circ  \circlearrowleft \hat{\mu}] (1 - \mathbb{L})^{|I| - 1}.
\end{equation}

 \subsection{Motivic invariants}\label{s:invariants} The importance of the motivic nearby fiber comes from its specializations. In this section, we will explain the following commutative diagram,
 which will be crucial in the rest of the paper, where the first vertical map is the motivic nearby fiber:
 \[\xymatrix{
K_0(\Var_\K) \ar[r]^{\bar{E}_{\hat{\mu}}} \ar[d]^\psi  &\Z[\Q/\Z][u,v,w]  \ar[r]^{\substack{u \mapsto uw^{-1} \\ v \mapsto 1}}   \ar[d]^{w \mapsto 1}  & \Z[\Q/\Z][u,w] \ar[d]^{w \mapsto 1}  &\\
K_0(\Var_\C)^{\hat{\mu}} \ar[r]^{E_{\hat{\mu}}} & \Z[\Q/\Z][u,v] \ar[r]^{v \mapsto 1} & \Z[\Q/\Z][u]  \ar[r]^{u \mapsto 1}  & \Z[\Q/\Z].
}\]

Throughout the paper, we will identify the group  $\Q/\Z$ with the group $\mathbb{S}^1_\Q$ of rational points on the circle $\{ z \in \C \mid |z| = 1 \}$, sending $[k] \in \Q/\Z$ to $\alpha = e^{2 \pi \sqrt{-1}k} \in \mathbb{S}^1_\Q$. We will consider the group algebra $\Z[\Q/\Z]$.

\begin{remark}\label{r:conjugate}
 Consider the
 $\Z$-algebra involution on
 $\Z[\Q/\Z]$ defined by setting $\bar{[k]} = [-k]$, corresponding to complex conjugation on the circle. This extends coefficient-wise to an involution on a polynomial ring with coefficients in  $\Z[\Q/\Z]$, that we will
 refer to as \define{conjugation}.
\end{remark}

\begin{remark}\label{r:forget}
Consider the natural $\Z$-algebra homorphism $\Z[\Q/\Z] \rightarrow \Z, [k] \mapsto 1$. One may apply this homomorphism below to get invariants with coefficients in $\Z$ rather than $\Z[\Q/\Z]$.
The corresponding diagram of invariants above over $\Z$ is explained in detail in \cite[Section~3]{Geometry}.
\end{remark}



Consider a complex vector space $B$ that admits a mixed Hodge structure \cite{PSMixed} with corresponding vector space decomposition
\[
B \cong \bigoplus_{p,q} H^{p,q}(B).
\]
Suppose the $\mu_n$ acts linearly on $B$, preserving the mixed Hodge structure. 
With the convention above, for  $\alpha \in \Q/\Z$, 
we write $h^{p,q}(B)_\alpha$ for the dimension of the $\alpha$-eigenspace of $H^{p,q}(B)$,
and write $h^{p,q}(B,\hat{\mu}) := \sum_{\alpha \in \Q/\Z} 
h^{p,q}(B)_\alpha \alpha  \in \Z[\Q/\Z]$. For a sequence of such representations $B_\bullet = \{ B_m \mid m \ge 0 \}$, set
$e^{p,q}(B_\bullet,\hat{\mu} ) = \sum_{m} (-1)^m h^{p,q}(B_m, \hat{\mu}) \in \Z[\Q/\Z]$. Then the \define{equivariant Hodge polynomial} of $B_\bullet$ is defined by
\[
E(B_\bullet) = E(B_\bullet; u,v) = \sum_{p,q} e^{p,q}(B_\bullet,\hat{\mu})    u^p v^q \in \Z[\Q/\Z][u,v].
\]


In  \cite{DelTheory}, Deligne proved that the  $m^{\textrm{th}}$ cohomology group  with compact supports $H^m_c  (V)$ of a complex variety $V$
admits a canonical mixed Hodge structure with decreasing filtration $F^\bullet$ called the \define{Hodge filtration} and increasing filtration $W_\bullet$ called the
\define{Deligne weight filtration}. Given a good action of  $\hat{\mu}$ on $V$, we may set $B_m = H^{p,q}(H^m_c  (V))$ above. The corresponding equivariant Hodge polynomial
is denoted
$E_{\hat{\mu}}(V;u,v) \in \Z[\Q/\Z][u,v]$, and  is called the \define{equivariant Hodge-Deligne polynomial} of $V$. The invariants $h^{p,q}(H^m_c  (V))_\alpha$ are called
the \define{equivariant mixed Hodge numbers} of $V$.
We have a ring homomorphism called the
\define{equivariant Hodge-Deligne map}:
\[
E_{\hat{\mu}}: K_0^{\hat{\mu}}(\Var_k) \rightarrow  \Z[\Q/\Z][u,v], [V] \mapsto E_{\hat{\mu}}(V;u,v).
\]

As in Remark~\ref{r:forget}, we will also consider the corresponding invariants over $\Z$. In this case, we have the usual mixed Hodge numbers  $h^{p,q}(H^m_c  (V)) = \sum_{\alpha} h^{p,q}(H^m_c  (V))_\alpha$
and corresponding Hodge-Deligne polynomial $E(V;u,v) \in \Z[u,v]$.

\begin{remark}
More generally, for any finite group $G$ acting algebraically on  a complex variety $V$, we obtain a linear action of $G$ on $H^{p,q}(H^m_c  (V))$, and one can form the
equivariant Hodge-Deligne polynomial
\[
E_G(V;u,v) = \sum_{p,q} \sum_m (-1)^m H^{p,q}(H^m_c  (V)) u^p v^q \in R(G)[u,v],
\]
where $R(G)$ is the complex representation ring of $G$. This induces a ring homomorphism $\tilde{K}_0^{G}(\Var_\C) \rightarrow R(G)[u,v]$. This invariant was introduced and studied by the author in
\cite{Representations},  generalizing the notion of weight polynomial due to Dimca and Lehrer \cite{DLPurity}, and  the notion of   equivariant $\chi_y$-genus due to Cappell, Maxim and Shaneson \cite{CMSEquivariant}. In our case, $\Z[\Q/\Z]$ may be viewed as the direct limit associated to the 
restriction maps $R(\mu_{n}) \rightarrow R(\mu_{nd})$ induced by
 the epimorphisms $\mu_{nd} \rightarrow \mu_n, \zeta \mapsto \zeta^d$.

\end{remark}

Recall that a variety $X$ over $\K$ is naturally interpreted as a complex variety with a morphism $f:X\rightarrow \C^* \setminus \{ b_1, \ldots , b_r \}$, for some points $b_1,\ldots, b_r$.
Restricting $X$ to a family over a sufficiently small punctured complex disc centered at the origin, and fixing a fiber $X_{\gen} := f^{-1}(t)$ and counter-clockwise orientation around the disc, there exists a quasi-unipotent monodromy map $T = T_sT_u: H^m_c(X_{\gen}; \C) \rightarrow  H^m_c(X_{\gen}; \C)$,  where $T_s$ is semi-simple and $T_u$ is unipotent. Then the cohomology
groups $H_c^m(X_{\gen})$ admit a weight filtration $M_\bullet$ called the \define{monodromy weight filtration},
and $T_s$ acts preserving the filtrations $(F^\bullet,M_\bullet, W_\bullet)$.
 We will write $H_c^m(X_{\infty})$ to denote $H_c^m(X_{\gen})$ with the filtrations $(F^\bullet,M_\bullet,W_\bullet)$.
 The filtrations $(F^\bullet,M_\bullet)$ induce a mixed Hodge structure on $H_c^m (X_{\infty})$, and  the nilpotent operator $N = \log T_u$ is a morphism of mixed Hodge structures of type $(-1,-1)$.
 The corresponding  invariants $h^{p,q}(H_c^m (X_{\infty}))_\alpha$ are called
the \define{equivariant limit mixed Hodge numbers}, and the corresponding equivariant Hodge polynomial $E(X_\infty, \hat{\mu};u,v) \in \Z[\Q/\Z][u,v]$ is called \define{equivariant limit Hodge-Deligne polynomial}.
A deep result of Denef and Loeser states that this is a specialization of the motivic nearby fiber in the
 sense that  $E(X_\infty,\hat{\mu};u,v) = E_{\hat{\mu}}(\psi_X)$. As in Remark~\ref{r:forget},  we may consider the corresponding invariants over $\Z$ i.e. the limit mixed Hodge numbers and
  limit Hodge-Deligne polynomial.

Before proceeding, we recall the following standard linear algebra construction.

\begin{definition}\label{d:weight}
Let $A$ be a nilpotent linear operator on a finite dimensional vector space $V$ such that $A^{r + 1} = 0$. Then the \define{$A$-weight filtration centered at $r$} is the increasing filtration $\{ V_\bullet \}$ of $V$ by subspaces
\[
0 \subseteq V_0 \subseteq V_1 \subseteq \cdots \subseteq V_{2r} = V,
\]
uniquely determined
by
\begin{enumerate}
\item $A( V_k ) \subseteq V_{k - 2}$,
\item\label{e:hard} the induced map $A^k: \Gr_{r + k} V \rightarrow \Gr_{r - k} V$ is an isomorphism,
\end{enumerate}
for any non-negative integer $k$. Here we set $V_k = 0$ for $k < 0$. The filtration encodes the Jordan block structure of $A$. Explicitly,
for $1 \le k \le r + 1$, the number of Jordan blocks of size $k$ equals
\[
\dim \Gr_{r + 1 - k} V - \dim \Gr_{r -  1 - k} V.
\]
\end{definition}

 The monodromy weight filtration $M_\bullet$ has the following important geometric property:  for every non-negative integer $r$,  the induced filtration $M(r)_\bullet$ on $\Gr_r^W H_c^m (X_{\infty})$
 coincides with the $N(r)$-weight filtration centered at $r$, where $N(r)$ is the nilpotent operator  induced by $N = \log T_u$.
It follows from Definition~\ref{d:weight} that the monodromy weight filtration, together with the action of $T_s$, encodes the Jordan block structure of the monodromy operator $T$  acting on  $\Gr_r^W H_c^m (X_{\infty})$.
 Moreover,  $(F^\bullet,M_\bullet)$ induces
 a mixed Hodge structure on  $\Gr_r^W H_c^m (X_{\infty})$ that is invariant  under $T_s$, and  $N(r)$ is a morphism of mixed Hodge structures of type $(-1,-1)$. The associated equivariant mixed Hodge numbers
 \[
 h^{p,q}(\Gr_r^W H_c^m (X_{\infty}))_\alpha =  ( \Gr_F^p \Gr_{p + q}^M \Gr_r^W H_c^m (X_{\infty}))_\alpha
 \]
 are called the \define{equivariant refined limit mixed Hodge numbers},
 and the associated equivariant Hodge polynomial is the coefficient of $w^r$ in a polynomial $E(X_\infty, \hat{\mu};u,v,w) \in \Z[\Q/\Z][u,v,w]$ called the
 \define{equivariant refined limit Hodge-Deligne polynomial}. As in Remark~\ref{r:forget},  we have corresponding invariants over $\Z$, that were introduced in \cite[Section~3]{Geometry}, to where
 we refer the reader for further details. We have a ring homomorphism
 \[
 \bar{E}_{\hat{\mu}}: K_0(\Var_\K) \rightarrow \Z[\Q/\Z][u,v,w],
 \]
 \[
 [X] \mapsto E(X_\infty,\hat{\mu};u,v,w).
 \]
 It follows from the definitions that we recover the equivariant limit mixed Hodge numbers and the usual mixed Hodge numbers of $X_{\gen}$ via the specializations:
 \begin{equation*}\label{e:limitnumber}
 h^{p,q}(H_c^m(X_\infty))_\alpha =  \sum_r h^{p,q,r}(H_c^m(X_{\infty}))_\alpha,
\end{equation*}
\begin{equation}\label{e:number}
h^{p,r-p}(H_c^m(X_{\gen})) = \sum_q \sum_{\alpha \in \Q/\Z} h^{p,q,r}(H_c^m(X_{\infty}))_\alpha.
\end{equation}
Correspondingly, we have $E(X_\infty, \hat{\mu};u,v,1) = E(X_\infty, \hat{\mu};u,v)$, and $E(X_\infty, \hat{\mu};uw^{-1},1,w)$ restricts to the Hodge-Deligne polynomial $E(X_{\gen}; u,w) \in \Z[u,w]$.
The further specializations $E(X_{\gen}; u,1)$ and $E(X_{\gen}; 1,1)$ are the $\chi_y$-characteristic and Euler characteristic of $X_{\gen}$ respectively.

\begin{remark}\label{r:unimodal}
With the notation above, since $N(r)$ is a morphism of mixed Hodge structures of type $(-1,-1)$, the isomorphisms \eqref{e:hard} in Definition~\ref{d:weight} imply that
for $0 \le j \le r$, the sequence
$\{ h^{i+ j,i,r}(H_c^m(X_{\infty}))_\alpha \mid 0 \le i \le r -j \}$ is symmetric and unimodal.
\end{remark}

 \begin{remark}\label{r:symmetriesrefined}
 Since $H^{p,q}(H_c^m(X_\infty))$ and $H^{q,p}(H_c^m(X_\infty))$ are conjugate, and by  Remark~\ref{r:unimodal}, the
  refined limit mixed Hodge numbers satisfy the following symmetries:
\[
h^{p,q,r}(H_c^m(X_{\infty}))_\alpha = h^{r-q,r-p,r}(H_c^m(X_{\infty}))_\alpha  = h^{q,p,r}(H_c^m(X_{\infty}))_{\alpha^{-1}}  = h^{r-p,r-q,r}(H_c^m(X_{\infty}))_{\alpha^{-1}} .
\]
Hence the equivariant refined limit Hodge-Deligne polynomial satisfies the following symmetries with respect to the conjugation action on $\Z[\Q/\Z][u,v,w]$ described in Remark~\ref{r:conjugate}:
\[
E(X_{\infty}, \hat{\mu};u,v,w) = \bar{E(X_{\infty}, \hat{\mu};v,u,w)},
\]
\[
E(X_{\infty}, \hat{\mu};u,v,w) = \bar{E(X_{\infty}, \hat{\mu};u^{-1},v^{-1},uvw)}.
\]
\end{remark}

\begin{example}\label{e:trivial2}
If $V$ is a complex variety and $X = V \times_\C \K$, then $X$
may be regarded as a trivial family over $\D^*$. In this case, $N$ is identically zero, $M_\bullet$ coincides with the Deligne filtration $W_\bullet$, and 
$E(X_\infty,\hat{\mu};u,v,w)  = E(V;uw,vw)$.
\end{example}

\begin{example}
If $X$ is smooth and proper, then $\Gr_r^W H^m (X_{\gen}) = 0$ unless $m = r$. In this case, the monodromy weight filtration, together with the action of $T_s$, encodes the Jordan block decomposition of $T$.
In this case,
\[
E(X_\infty,\hat{\mu};u,v,w) =   \sum_{p,q,m} (-1)^m h^{p,q}(H^m(X_{\infty}))_\alpha u^p v^q w^m \alpha.
\]
\end{example}

\section{Tropical geometry}\label{s:trop}

In this section we prove our formula for the motivic nearby fiber. The non-equivariant version of this result is \cite[Theorem~1.2]{Geometry}.
We continue with the notation of the previous section. In particular, $\K = \C(t)$.

Let $w= (w_1, \ldots, w_n) \in \R^n$ and let $f \in \K[x_{1}^{\pm 1},\ldots, x_n^{\pm 1}]$. Then the  \define{initial degeneration}  $\init_w(f) \in \C[x_{1}^{\pm 1},\ldots, x_n^{\pm 1}]$ of $f$ with respect to $w$
is defined as follows: write $f = \sum_{u \in \Z^n} 
 \lambda_{u}  t^{\omega_f(u)} g_u(t) x^{\alpha}$
for some  $\lambda_{u} \in \C$,  $\omega_f(u) \in \Z$ and
$g_u(t) \in \C(t)$ such that $g(1) = 1$, and  let $m_f = \min \{ \omega_f(u) + w \cdot u \mid \lambda_{u} \ne 0 \}$. Then
\[
\init_w(f) = \sum_{ \substack{ u \in \Z^n \\  \omega_f(u) + w \cdot u = m_f  }}  \lambda_{u} x^{u}.
\]
Assume that $w$ is rational and consider the element
\[
\exp( 2\pi \sqrt{-1} w) := (e^{2\pi \sqrt{-1} w_1}, \ldots,e^{2\pi \sqrt{-1} w_n}) \in (\C^*)^n. \] Then multiplication by $\exp( 2\pi \sqrt{-1} w)$ induces a finite order automorphism $\beta_{w}$ of $ (\C^*)^n$
and hence a good action of $\hat{\mu}$ on $(\C^*)^n$. Note that if $w \ne 0$, then $\beta_{w}$ acts freely on $ (\C^*)^n$.   We compute
\begin{equation}\label{e:pullback}
(\beta_{w})^*(\init_w(f) ) = \init_w(f)(e^{2\pi \sqrt{-1} w_1}x_1, \ldots, e^{2\pi \sqrt{-1} w_1}x_n) = e^{2\pi \sqrt{-1} m_f} \init_w(f).
\end{equation}

Let $X^\circ \subseteq (\K^*)^n$ be a closed subvariety defined by an  ideal $I \subseteq \K[x_{1}^{\pm 1},\ldots, x_n^{\pm 1}]$. Then the initial degeneration $\init_w X^\circ$ is the closed
subscheme of $(\C^*)^n$ defined by the ideal  $\init_w I := ( \init_w(f) \mid f \in I) \subseteq \C[x_{1}^{\pm 1},\ldots, x_n^{\pm 1}]$. By \eqref{e:pullback}, $\init_w I$ is invariant under $\beta_{w}$,
and hence we have an induced action on   $\init_w X^\circ$.


A variety $X$ over $\K$ is \define{very affine} if it admits a closed embedding $X \subseteq (\K^*)^n$ for some $n$.
Tevelev introduced the notion
of a \define{sch\"on}, very affine variety. 
A closed subvariety $X^\circ \subseteq (\K^*)^n$ is sch\"on if $\init_w X^\circ \subseteq (\C^*)^n$ is a smooth (reduced) subvariety for every $w \in \R^n$ \cite[Prop 3.9]{HelmKatz}.
A very affine variety $X$ over $\K$ is sch\"on if it admits a sch\"on closed embedding $X \subseteq (\K^*)^n$ for some $n$ (cf. \cite[Lemma~2.11]{LuxtonQu}).
Note that a complex variety $V$ may be viewed via base change as a variety over $\K$, and hence we may consider the same notions. The definition of sch\"on for a hypersurface of a complex torus
agrees with the notion of   \emph{non-degeneracy} with respect to the Newton polytope  \cite{KhoNewton}.

The \define{tropical variety} $\Trop(X^\circ)$ associated to $X^\circ \subseteq (\K^*)^n$  is the set of points 
\[\{w\in\R^n \mid \init_w X^\circ\neq \emptyset\}.\]
The tropical variety $\Trop(X^\circ)$ can be given a rational polyhedral structure $\Sigma$ such that initial degeneration at $w \in \Trop(X^\circ)$ only depends on the cell containing $w$ in its relative interior (this follows from
\cite[Theorem~1.5]{LuxtonQu}). Hence
for every cell $\gamma$ of $\Sigma$, we may define $\init_{\gamma} X^\circ := \init_w X^\circ$ for any
$w \in \R^n$ in the relative interior 
of $\gamma$. Moreover, by fixing such a $w \in \Q^n$, we obtain a good action of $\hat{\mu}$ on  $\init_{\gamma} X^\circ$ via multiplication by $\exp( 2\pi \sqrt{-1} w)$.

\begin{example}\label{e:simple}
Consider the variety $X^\circ = \{ \sum_{i  = 1}^n 
x_i^{m_i} = t^e \} \subseteq (\K^*)^n$, for some  
$m_i \in \Z_{>0}$ and $e \in \Z$. Let $e_1, \ldots, e_n$ denote
the standard basis vectors of $\R^n$. Then $X^\circ \subseteq (\K^*)^n$ is sch\"on and $\Trop(X^\circ) \subseteq \R^n$  is the affine translation  by $(e/m_1, \ldots, e/m_n)$ of the $(n-1)$-dimensional fan with
cones spanned by subsets of size at most $n - 1$ of $\{e_1, \ldots, e_n, -e_1/m_1 - \cdots -e_n/m_n \}$. Let $w = (e/m_1, \ldots, e/m_n)$. Then $\hat{\mu}$ acts on
$\init_w X^\circ =  \{ \sum_{i  = 1}^n 
x_i^{m_i} = 1 \} \subseteq (\C^*)^n$ via multiplication by
$(e^{2\pi \sqrt{-1} e/m_1}, \ldots,e^{2\pi \sqrt{-1} e/m_n})$. Considering $X^\circ$ as a family $f: X^\circ \rightarrow \A^1$, for $0 \le a \le 1$, multiplication by $(e^{2\pi \sqrt{-1} ae/m_1}, \ldots,e^{2\pi \sqrt{-1} ae/m_n})$
induces an isomorphism between  $f^{-1}(1)$ and $f^{-1}(e^{2\pi \sqrt{-1} a})$. We deduce that the action of $\hat{\mu}$ coincides with the action of monodromy
induced by moving anti-clockwise
around the unit circle
(more generally, see Example~\ref{e:simpleaffine}).

\end{example}

Recall that our goal is to compute the motivic nearby fiber defined in Section~\ref{s:mnf}:
\[
\psi : K_0(\Var_\K) \rightarrow K_0^{\hat{\mu} }(\Var_\C),
\]
\[
[X] \mapsto \psi_X.
\]
A result of Luxton and Qu  \cite[Theorem~6.11]{LuxtonQu} that was conjectured by Tevelev in \cite{TevComp},  implies that every variety over $\K$ admits
a stratification into locally closed, very affine, sch\"on subvarieties. By the additivity property \eqref{e:cut} of $K_0(\Var_\K)$, it follows that we may reduce the computation
of $\psi$ to the case when   $X^\circ \subseteq (\K^*)^n$ is a sch\"on subvariety. In this case, we have the following equivariant  generalization of \cite[Theorem~1.2]{Geometry}.


\begin{theorem} \label{t:comp} Let $X^\circ \subseteq (\K^*)^n$ be a sch\"{o}n closed subvariety and let $\Sigma$ be a rational polyhedral structure on $\Trop(X^\circ)$.
Then the motivic nearby fiber $\psi_{X^\circ} \in K_0^{\hat{\mu} }(\Var_\C)$ is given by
\[
\psi_{X^\circ} =   \sum_{\substack{\gamma\in\Sigma\\\gamma \operatorname{bounded}}}\, (-1)^{\dim \gamma} [\init_\gamma X^\circ \circlearrowleft \hat{\mu}],
\]
where the action of $\hat{\mu}$ on $\init_\gamma X^\circ \subseteq (\C^*)^n$ is induced by multiplication by $\exp( 2\pi \sqrt{-1} w) =
(e^{2\pi \sqrt{-1} w_1}, \ldots,e^{2\pi \sqrt{-1} w_n})$ for a fixed choice of $w \in \Q^n$ in the relative interior of $\gamma$.
\end{theorem}
\begin{proof}
We first recall the relevant geometric setting, as described in \cite[Sections 1,2]{HelmKatz}. We refer the reader to \cite{FulIntroduction} for the appropriate background on toric geometry.

Let $N = \Z^n$,  and 
for each cell $\gamma \in \Sigma$, let $C_\gamma$ denote the cone generated by $\gamma \times \{ 1\}$ in $N_\R \times \R$.
The cones $C_\gamma$ form a fan $\tilde{\Sigma}$ in $N_\R \times \R$, and the recession fan $\Delta$ is defined to be the  intersection of $\tilde{\Sigma}$ with $N_\R \times \{ 0\}$. Projection
$\pr: N_\R \times \R \rightarrow \R$ onto the second coordinate induces a morphism of fans $\tilde{\Sigma} \rightarrow \R_{\ge 0}$.  Fix a positive integer $m$ such that $m\Sigma$ has a lattice polyhedral structure, and let $\tilde{\P}$ and $\P$ denote the complex toric varieties corresponding to
the fan $\tilde{\Sigma}$ with respect to the lattices $N \times m\Z$ and $N \times \Z$ respectively. 
We have an induced commutative diagram corresponding to pulling back and normalizing $\P$:
\[\xymatrix{
\tilde{\P}  \ar[d] \ar[r]^{\pi}  & \P  \ar[d] \\
\A^1 \ar[r]^{t \mapsto t^m}  & \A^1, \\
}\]
where the vertical arrows are flat. After removing the closed torus invariant divisors corresponding to non-zero cones in $\Delta$, the non-zero fibers of  $\tilde{\P}$ and $\P$ are tori and the central fibers
have decompositions into disjoint unions of tori:
\[
\tilde{\P}^\circ = \bigcup_{\substack{\gamma\in\Sigma\\\gamma \operatorname{bounded}}} \tilde{U}_\gamma, \; \; \; \P^\circ = \bigcup_{\substack{\gamma\in\Sigma\\\gamma \operatorname{bounded}}} U_\gamma.
\]
Let $\pi_\gamma: \tilde{U}_\gamma \rightarrow U_\gamma$ denote the morphism of tori induced by $\pi$. This has the following explicit description.
Let $L_\gamma'$ denote the linear span of $C_\gamma$, and let $L_\gamma = L_\gamma' \cap (N_\R \times \{ 0 \}) \subseteq N_\R$. That is, $L_\gamma$ is the translate of the affine span of $\gamma$ to the origin
in $N_\R$. Fix $m_\gamma \in \Z_{>0}$  such that
 $\pr( (N \times \Z) \cap  L_\gamma') = m_\gamma \Z$. Let $N_\gamma = N / N \cap L_\gamma$ with distinguished rational point $[\gamma] \in (N_\gamma)_\Q$. Then projection along
 $L_\gamma'$ induces an isomorphism $N_\Q \times \Q / (N_\Q \times \Q  ) \cap L_\gamma' \rightarrow (N_\gamma)_\Q$ that sends $(0,1) = ([-\gamma],0) + ([\gamma],1) \mapsto [-\gamma]$.
 This isomorphism induces the vertical maps in the following  commutative diagrams:
\[\xymatrix{
0 \ar[r] & N \times m\Z / (N \times m\Z ) \cap L_\gamma'   \ar[d]^{\cong} \ar[r]  & N \times \Z / (N \times \Z ) \cap L_\gamma'   \ar[d]^{\cong} \ar[r]  & \Z/m_\gamma \Z   \ar[d]^{=} \ar[r] & 0\\
0 \ar[r] &  N_\gamma      \ar[r]    & N_\gamma + \Z \cdot [-\gamma] \ar[r]  &\Z/m_\gamma \Z \ar[r]  &  0  \\
}\]
 Applying the functors $\Hom(\underline{\hspace{3mm}}, \Z)$ and then $\Hom(\underline{\hspace{3mm}}, \C^*)$ induces
 \[
 1 \rightarrow \mu_{m_\gamma} \rightarrow \tilde{U}_\gamma \xrightarrow{\pi_\gamma} U_\gamma \rightarrow 1.
 \]
 Letting $M = \Hom(N, \Z)$ and $M_\gamma = \Hom(N_\gamma,\Z)$, we identify $\tilde{U}_\gamma = \Hom(M_\gamma, \C^*)$. Then the above sequence says that the rational point $[-\gamma] \in (N_\gamma)_\Q$ corresponds to
 an element  $\exp(2\pi \sqrt{-1} [-\gamma]) \in \tilde{U}_\gamma$ of order $m_\gamma$
 such that   $U_\gamma$ is the quotient of $\tilde{U}_\gamma$ by the action of  multiplication by $\exp(2\pi \sqrt{-1} [-\gamma])$.
 In what follows, we let $T = \Hom(M, \C^*)$ and consider
the multiplication map $T \times \tilde{U}_\gamma \rightarrow \tilde{U}_\gamma$. We also fix the distinguished element $p \in \tilde{U}_\gamma= \Hom(M_\gamma, \C^*)$ corresponding to the constant map $1$.

Recall that $X^\circ$ may  viewed as a family of complex varieties over $\C^* \setminus \{ b_1, \ldots , b_r \}$, for some points $b_1,\ldots, b_r$. Let $\D$ be a sufficiently small complex disc centered at the origin, and
restrict $X^\circ$ to a family over the punctured disc $\D^*$, and $\tilde{\P}$ and $\P$ to families over $\D$. Taking the closure $\mathcal{X}$ of $X^\circ$ in $\P$ induces a commutative diagram:
\[\xymatrix{
\tilde{\mathcal{X}} \subseteq \tilde{\P}  \ar[d] \ar[r]^{\pi}  & \mathcal{X} \subseteq \P  \ar[d] \\
\D \ar[r]^{t \mapsto t^m}  & \D, \\
}\]
where for each bounded cell $\gamma \in \Sigma$, we have a corresponding $\mu_{m_\gamma}$-covering
 $\tilde{V}^\circ_\gamma := \tilde{\mathcal{X}} \cap \tilde{U}_\gamma \rightarrow V^\circ_\gamma := \mathcal{X} \cap U_\gamma$ corresponding to the action of multiplication by $\exp(2\pi \sqrt{-1} [-\gamma])$.
As above, consider the induced multiplication map $m: T \times \tilde{V}^\circ_\gamma \rightarrow  \tilde{U}_\gamma$ and consider the fiber $m^{-1}(p)$. By \cite[p. 8]{HelmKatz}, projection onto
$T$ identifies $m^{-1}(p)$ with $\init_\gamma X^\circ$, while projection onto $\tilde{V}^\circ_\gamma$ identifies $m^{-1}(p)$ with $T_\gamma \times  \tilde{V}^\circ_\gamma$, where $T_\gamma$ is
the subtorus that fixes $U_\gamma$ pointwise. Fix a rational point $w \in N_\Q$ in the relative interior of $\gamma$, and consider the corresponding point $\exp(2\pi \sqrt{-1} w) \in T$ of finite order. Then
multiplication by $(\exp( 2\pi \sqrt{-1} w), \exp(2\pi \sqrt{-1} [-\gamma])$ preserves  $m^{-1}(p)$, and, using Example~\ref{e:torus}, we conclude
that
\begin{equation}\label{e:different}
[\init_\gamma X^\circ \circlearrowleft \hat{\mu}] = [ \tilde{V}^\circ_\gamma  \circlearrowleft \hat{\mu}] (\L - 1)^{\dim \gamma},
\end{equation}
where $\hat{\mu}$ acts via multiplication by $\exp(2\pi \sqrt{-1} w)$ and $\exp(2\pi \sqrt{-1} [-\gamma])$ respectively. In particular, this shows that the left hand side is independent of the choice of $w$ in the relative
interior of $\gamma$.

We now complete the proof. We first verify that the right hand side of the statement in the theorem is independent of the choice of polyhedral structure $\Sigma$.
Indeed, if $\Sigma'$ is a polyhedral structure refining $\Sigma$, then, using the independence of the choice of $w$ above
\[
\sum_{\substack{\gamma' \in\Sigma' \\\gamma \operatorname{bounded}}}\, (-1)^{\dim \gamma'} [\init_{\gamma'} X^\circ \circlearrowleft \hat{\mu}] =
\sum_{\gamma \in \Sigma} (-1)^{\dim \gamma} [\init_{\gamma} X^\circ \circlearrowleft \hat{\mu}]  \left(\sum_{\substack{\gamma' \in \Sigma', \gamma' \operatorname{bounded}\\ \Int(\gamma') \subseteq \Int(\gamma) }} (-1)^{\dim \gamma - \dim \gamma'}\right).
\]
Hence, it is enough to show that
\[
 \sum_{\substack{\gamma' \in \Sigma', \gamma' \operatorname{bounded}\\ \Int(\gamma') \subseteq \Int(\gamma) }} (-1)^{\dim \gamma'}  =  \left\{\begin{array}{cl} (-1)^{\dim \gamma} & \text{if } \gamma \text{ is bounded } \\ 0 & \text{otherwise}. \end{array}\right.
\]
Topological and combinatorial proofs of the above fact are given in the proofs of \cite[Lemma~2.2]{Geometry} and \cite[Theorem~3.4]{KatzStapledon} respectively.

After possibly replacing $\tilde{\Sigma}$ with a smooth fan refinement, and replacing $\Sigma$ with $\tilde{\Sigma} \cap (N_\R \times \{ 1 \})$, we may assume that $\P$ is smooth and the central fiber $\mathcal{X}_0$
is a simple normal crossings divisor. 
Let $\{v_1, \ldots, v_s \}$ denote the vertices of $\gamma$. Then $\mathcal{X}_0$ has irreducible components $\{ D_1, \ldots , D_s \}$ with multiplicities $\{ m_1, \ldots, m_s \}$ respectively, where
$m_i$ is the smallest positive integer $m$ such that $mv_i \in N$. For every non-empty subset $I \subseteq \{ 1,\ldots s \}$, let $\gamma_I$ be the bounded cell of $\Sigma$ given by the
convex hull of $\{ v_i \mid i \in I \}$. Observe that every bounded cell appears in this way. As above, $D_I^\circ = \cap_{i \in I} D_i \setminus \cup_{j \notin I} D_j = V^\circ_\gamma$ admits an $\mu_{m_\gamma}$-covering
$\tilde{V}^\circ_\gamma  \rightarrow V^\circ_\gamma$ corresponding to the action of multiplication by $\exp(2\pi \sqrt{-1} [-\gamma])$. It follows from the fact that $C_\gamma$ is a smooth cone in $N \times \Z$ that $m_\gamma = \gcd(m_i \mid i \in I )$. The result now follows from \eqref{e:different} above together with the definition of the motivic nearby fiber given in \eqref{e:dmotivic}.
\end{proof}

\begin{remark}
It is a corollary of the proof above that in the statement of Theorem~\ref{t:comp}, one may consider  the action on $\init_\gamma X^\circ \subseteq (\C^*)^n$ to be induced by multiplication by $\exp(2\pi \sqrt{-1} w)$ for any fixed choice of $w \in \Q^n$ in the affine span of $\gamma$.
\end{remark}

We will see interesting examples of the above theorem in Section~\ref{s:hyper} and Section~\ref{s:sing}.
For the moment, we have the following simple example.

\begin{example}\label{e:simple2}
Following on with Example~\ref{e:simple}, let $X^\circ = \{ \sum_{i  = 1}^n 
x_i^{m_i} = t^e \} \subseteq (\K^*)^n$ for some  
$m_i \in \Z_{>0}$ and $e \in \Z$. Then $\psi_{X^\circ} = [\{ \sum_{i  = 1}^n 
x_i^{m_i}  = 1 \} \subseteq (\C^*)^n \circlearrowleft \hat{\mu}]$, with $\hat{\mu}$-action
induced  by multiplication by $(e^{2\pi \sqrt{-1} e/m_1}, \ldots,e^{2\pi \sqrt{-1} e/m_n})$. If $X=  \{ \sum_{i  = 1}^n 
x_i^{m_i} = t^e \} \subseteq \K^n$  denotes the closure of $X^\circ$ in $\K^n$,
then  $\psi_{X} = [\{ \sum_{i  = 1}^n 
x_i^{m_i}  = 1 \} \subseteq \C^n \circlearrowleft \hat{\mu}]$, with action again given by multiplication by $(e^{2\pi \sqrt{-1} e/m_1}, \ldots,e^{2\pi \sqrt{-1} e/m_n})$.
As in  Example~\ref{e:simple}, the above action coincides with the action induced by monodromy.

\end{example}


\section{The combinatorics of height functions
}\label{s:weighted}

In this section, we present the relevant combinatorics of this paper. Weighted Ehrhart theory was introduced by the author in  \cite{Weighted} both in order to  extend and reprove many results in
Ehrhart theory, the study of enumeration of lattice points in polytopes, and in order to give explicit computations of  motivic integrals on toric stacks \cite{Stacks}. We present a generalization of this theory below using
an extension of Stanley's theory of subdivisions \cite{StaSubdivisions} presented in \cite{Combinatorics}. For brevity, we will only present the relevant results and refer the reader to \cite{Combinatorics}  and  \cite{Weighted} for details of proofs.




\subsection{Eulerian posets}

We first recall some basic notions on Eulerian posets.
Consider a finite poset  $B$  containing a minimal element $\hat{0}$ and a maximal element $\hat{1}$. For any pair $z \leq x$ in $B$, we can consider the interval $[z,x] = \{ y \in B \mid z \leq y \leq x \}$.
Assume that $B$ is \emph{graded} in the sense that for every $x \in B$, every maximal chain in
 the interval $[\hat{0},x]$ has the same length $\rho(x)$. We call $\rho: B \rightarrow \N$ the \emph{rank function} of $B$, and call $\rho(\hat{1})$ the \emph{rank} of $B$. Then $B$ is \define{Eulerian} if every interval $[z,x]$ with $z < x$ has as many elements of odd rank as even rank.

\begin{example}\label{e:polytope}
The poset of faces of a polytope $P$ (including the empty face) is an Eulerian poset under inclusion,
with $\rho(Q) = \dim Q + 1$, for any face $Q$ of $P$. \emph{Throughout this paper, the empty polytope has dimension $-1$. }
\end{example}

\begin{example}
If $B$ is a poset, then $B^*$ is the poset with the same elements as $B$ and all orderings reversed. In particular, $B$ is Eulerian if and only if $B^*$ is Eulerian.
\end{example}

The \define{$g$-polynomial} of an Eulerian poset is defined recursively and was introduced by Stanley \cite[Corollary~6.7]{StaSubdivisions}.

\begin{definition}\label{d:g}
Let $B$ be an Eulerian poset of rank $n$. If $n = 0$, then $g(B;t) = 1$. If $n > 0$, then $g(B;t)$ is the unique polynomial
of degree strictly less the $n/2$ satisfying
\[
t^{n}g(B;t^{-1}) = \sum_{x \in B} (t - 1)^{n - \rho(x)} g([\hat{0},x];t).
\]
\end{definition}

\begin{example}\label{e:linear}
By considering the leading terms above, we see that $g(B;t)$ has constant term $1$, and  linear term $\# \{ x \in B \mid \rho(x) = 1 \} - n$.
\end{example}

\begin{example}\label{e:boolean}
One can verify that $g(B;t) = 1$ if and only if $B$ is the Boolean algebra on $r$ elements for some $r$ \cite[Remark~4.2]{BNCombinatorial}.
\end{example}

We will need  the following theorem of Stanley.

\begin{theorem}   \label{t:Eulerianinverse}\cite[Corollary~8.3]{StaSubdivisions}
If $B$ is Eulerian and has positive rank, then
\[
\sum_{x \in B} (-1)^{\rho(x)} g([\hat{0},x];t)g([x,\hat{1}]^*;t) = \sum_{x \in B} (-1)^{\rho(x)} g([\hat{0},x]^*;t)g([x,\hat{1}];t)  =  0.
\]
\end{theorem}

\subsection{Polyhedral subdivisions}\label{subdivisions}

A \define{polyhedral subdivision} of a polytope $P\subseteq\R^n$ is a subdivision of $P$ into a finite number of polytopes such that the intersection of any two polytopes is a (possibly empty) face of both.
A polyhedral subdivision is a rational  (respectively lattice)  subdivision if each polytope has rational (respectively integral) vertices.
A natural class of lattice polyhedral subdivisions are the \define{regular} lattice subdivisions.   They are induced by a function $\omega: S \rightarrow \Z$, called an \define{integral height function} , for some finite subset $S \subseteq P \cap \Z^n$ containing the vertices of $P$.
The cells of the subdivision are the projections
of the bounded faces of the convex hull $\operatorname{UH}$
of
$\{ (u, \lambda) \mid u \in S,  \lambda \geq\omega(u)  \} \subseteq \R^n \times \R$.  The bounded faces of $\operatorname{UH}$ form the graph
of a function
$\nu: P \rightarrow \R$, called the  \define{convex graph} of $\omega$.  Note that $\nu|_S$ is not necessarily the same as $\omega|_S$.
Let $\cS(\nu)$ denote the regular lattice polyhedral subdivision induced by $\omega$. Then $\nu$ is convex and piecewise $\Q$-affine with respect to $\cS(\nu)$, and takes integer values on the vertices of $\cS(\nu)$.
Conversely, if $\nu: P \rightarrow \R$ is a  convex function that is  piecewise $\Q$-affine with respect to a lattice polyhedral subdivision  $\cS$, such that $\nu$ takes integer values on the vertices of $\cS$, then
the restriction of $\nu$ to the vertices of $\cS$  is an integral height function such that $\nu$ is the corresponding convex graph and $\cS$ is a refinement of $\cS(\nu)$.
 For more details, including the general definition of a regular polyhedral subdivision,  see \cite{Triangulations,GKZ}.

By abuse of notation, we write $\cS$ for the poset of faces (including the empty face) of a 
polyhedral subdivision $\cS$ of $P$. Let $[\emptyset, P]$ denote the face poset of $P$. In what follows, we will consider the
function
\[
\sigma: \cS \rightarrow [\emptyset, P],
\]
where $\sigma(\emptyset) = \emptyset$, and for a non-empty cell $F \in \cS$,  $\sigma(F)$ is the smallest face of $P$ containing $F$.
As a poset, the \define{link} $\lk_\cS(F)$  of $F$ in $\cS$ consists of all
cells $F'$ of $\cS$ that contain $F$. 

We will need the following two closely related polynomials.
In the case $F = \emptyset$,  the 
local $h$-polynomial defined below  was   introduced by Stanley in \cite[Corollary~7.7]{StaSubdivisions}.
In the case of a triangulation of a simplex, it was introduced by Athanasiadis \cite{AthFlag, AthSymmetric}, and the general case of a polyhedral subdivision of a polytope was first considered by Nill and Schepers  \cite{NSCombinatorial}, where it was called the `relative local $h$-polynomial'.
A more general definition of the 
local $h$-polynomial appears in \cite[Definition~4.1]{Combinatorics}.

\begin{definition}
Let $\cS$ be a 
polyhedral subdivision of a 
polytope $P$. If $F$ is a  (possibly empty) cell of $\cS$, then the \define{$h$-polynomial} of $\lk_\cS(F)$ is defined by
\[
t^{\dim P - \dim F}h(\lk_\cS(F);t^{-1}) = \sum_{\substack{F' \in \cS \\\ F \subseteq F'}} g([F,F'];t)(t - 1)^{\dim P - \dim F'}.
\]
The  \define{
local $h$-polynomial} of $F$ in $\cS$ is the polynomial
\[
l_P(\cS,F;t)
= \sum_{\sigma(F) \subseteq Q \subseteq P} (-1)^{\dim P - \dim Q} h(\lk_{\cS|_Q}(F);t) g([Q,P]^*;t).
\]
When $F = \emptyset$, we write $h(\cS;t)  = h(\lk_\cS(F);t)$ and $l_P(\cS;t) = l_P(\cS,F;t)$.
\end{definition}

\begin{example}\label{e:trivialdecomp}
 If $\cS$ is the trivial subdivision of $P$, with cells of $\cS$ given by the faces of $P$, then one can verify that
 $h(\lk_{\cS}(Q);t) = g([Q,P];t)$ and 
 $l_P(\cS,Q;t) = 0$ whenever $Q$ is a (possibly empty) proper face of $P$.
\end{example}

We state some properties of these polynomials below. 
The non-negativity of the $g$-polynomial and $h$-polynomial and the 
local $h$-polynomial when $F = \emptyset$  is due to Stanley \cite[Theorem~5.2,Theorem~7.9]{StaSubdivisions}.

\begin{remark}\label{r:non-negative}\cite[Theorem~6.1]{Combinatorics}
Let $\cS$ be a rational polyhedral subdivision of a lattice polytope $P$. For any cell $F \in \cS$ and face $Q$ of $P$, the polynomials $g([Q,P]^*;t)$,
$g([Q,P];t)$, $h(\lk_\cS(F);t)$ and $l_P(\cS,F;t)$ have non-negative coefficients. We have the symmetry
\[
l_P(\cS,F;t) = t^{\dim P - \dim F} l_P(\cS,F;t^{-1}).
\]
If $\cS$ is a regular subdivision, then the coefficients of $l_P(\cS,F;t)$ are symmetric and unimodal.
\end{remark}

In order that one may compute examples, we list some of the coefficients below (see \cite[Example~3.13,Remark~3.32,Example~4.9,Example~4.10]{Combinatorics}).

\begin{example}\label{e:relative}
Let $\cS$ be a 
polyhedral subdivision of a 
polytope $P$. Let $F$ be a cell of $\cS$. Let
\[
\lambda =  \# \{ F \subseteq F' \mid \dim F' = \dim F + 1 \},
\]
\[
\lambda' =  \# \{ F \subseteq F' \mid \dim F' = \dim F + 1, \sigma(F') = P \}.
\]
The constant term of $h(\lk_\cS(F);t)$ is $1$, and the linear coefficient is $\lambda - (\dim P - \dim F)$.
 If $\sigma(F) = P$, then $h(\lk_\cS(F);t) = l_P(\cS,F;t)$
is symmetric. Assume that $\sigma(F) \ne P$. Then $h(\lk_\cS(F);t)$ has degree at most $\dim P - \dim F - 1$,
and the coefficient of $t^{\dim P - \dim F - 1}$ is  $\lambda'$. Moreover, $l_P(\cS,F;t)$ is symmetric with constant term $0$ and
linear term
\[
   \left\{\begin{array}{cl} \lambda' - 1 & \text{if } \dim P - \dim \sigma(F) = 1 \\ \lambda' & \text{if } \dim P - \dim \sigma(F) > 1 . \end{array}\right.
\]
\end{example}

\subsection{Weighted Ehrhart theory}

Fix a lattice polytope $P$ with respect to a lattice $M$. After possibly replacing $M$ with a sublattice, we may assume that $\dim P = \dim M_\R$.
Under an isomorphism $M \cong \Z^n$, the \define{normalized volume} $\Vol(P) \in \Z_{> 0}$ of $P$ with respect to $M$ is $n!$ times the Euclidean volume of $P$.
Let $\nu: P \rightarrow \R$ be the convex graph of an integral height function on $P$, with corresponding regular lattice polyhedral subdivision $\cS(\nu)$ of $P$.

\begin{remark}\label{r:old}
The setup above is slightly different than the `weight $0$' setup in \cite{Weighted}, although the same arguments hold.  We explain the relationship below.
In  \cite{Weighted}, we consider a stacky fan $(N,\Sigma,\{ b_i \})$, where $\Sigma$ is a full-dimensional simplicial fan with respect to a lattice $N$, and $\{ b_i \}$ is a choice of a non-zero lattice
point on each ray of $\Sigma$. We let  $\psi$ be the piecewise $\Q$-linear function with respect to $\Sigma$ satisfying $\psi(b_i) = 1$,   and let $Q = \{ v \in N_\R \mid \psi(v) \le 1 \}$.
If $Q$ is a convex polytope, then we may set $M = N$, $P = Q$ and $\nu = \psi$ above to recover this setup.
Here $\Sigma$ induces a triangulation $\cS$ of $P$ refining $\cS(\nu)$ such that  $\nu$ takes integer values on the vertices of $\cS$. We also note that the definition of the weighted $h^*$-polynomial below is slightly different (although equivalent) to that in \cite{Weighted} (see
Example~\ref{e:orbifold}).
\end{remark}




Let $C$ be the cone over $P \times \{ 1\}$
in $M_\R \times \R$,  and
extend $\nu$ by linearity to a piecewise $\Q$-linear function on $C$.  For a non-negative integer $m$, we will often identify $mP$ with $C \cap (M_\R \times \{ m \})$ and $M_\R$ with $M_\R \times \{ m \}$.
We define a \define{weight function}
\[
w: C \cap (M \times \Z) \rightarrow \Q/\Z,
\]
\[
w(v) = [\nu(v)].
\]
Note that adding a global affine function with respect to $M$ to $\nu$ does not change  $w$.
Recall from Remark~\ref{r:conjugate} that the group algebra $\Z[\Q/\Z]$ admits a conjugation action such that $\bar{[k]} = [-k]$, and a `forgetful' $\Z$-algebra homorphism $\Z[\Q/\Z] \rightarrow \Z, [k] \mapsto 1$.
The \define{weighted Ehrhart polynomial} is defined by
\[
f(P,\nu;m) = \sum_{v \in mP \cap M } 
w(v) \in \Z[\Q/\Z],
\]
for  every non-negative integer $m$. 
As in Remark~\ref{r:forget}, the corresponding polynomial with coefficients in $\Z$ is the usual \define{Ehrhart polynomial} of $P$  \cite{Ehr1,Ehr2}, defined by
$f(P;m) = \#(mP \cap M )$  for every non-negative integer $m$. By  \cite[Theorem~3.7]{Weighted}, $f(P,\nu;m)$ is a polynomial in $m$ of degree $\dim P$, and  for every positive integer $m$,
\[
(-1)^{\dim P} f(P,\nu;-m) = \sum_{v \in \Int(mP) \cap M} \bar{w(v)} \in \Z[\Q/\Z].
\]
 The latter result is known as \define{weighted Ehrhart reciprocity}, since it reduces to the classical Ehrhart reciprocity: $f(P;-m) = \#(\Int(mP) \cap M )$ for every positive integer $m$ \cite{EhrReciprocity}.
 In fact, it is shown in \cite[Remark~3.10]{Weighted} that weighted Ehrhart reciprocity generalizes Ehrhart reciprocity for rational polytopes. Let $\cS$ be a lattice polyhedral decomposition refining
 $\cS(\nu)$  such that $\nu$ takes integer values on the vertices of $\cS$.  
 For each maximal cell $F \in \cS$, let $m_F$ be the minimal positive integer such that  $m_F \nu|_{F}$ is an affine 
 function with respect to $M$. Then one verifies that the
 leading coefficient of $f(P,\nu;m)$ as a polynomial in $m$ is given by:
 \begin{equation}\label{e:leading}
\frac{1}{(\dim P)!} \sum_{ \substack{F \in \cS \\ \dim F = \dim P} }  \frac{\Vol(F)}{m_F} \sum_{i = 0}^{m_F - 1} [i/m_F] \in \Z[\Q/\Z],
 \end{equation}
 where $\Vol(F)$ is the normalized volume of $F$ and $\Vol(F)/m_F$ is a positive integer.
Define the \define{weighted $h^*$-polynomial} $ h^*(P,\nu;u) \in \Z[\Q/\Z][u]$ by
\[
\sum_{m \ge 0} f(P,\nu;m) u^m = \frac{ h^*(P,\nu;u)  }{(1 - u)^{\dim P + 1}}.
\]
When $P$ is the empty polytope, we set $h^*(P,\nu;u) = 1$.
As in Remark~\ref{r:forget}, the corresponding  polynomial with coefficients in $\Z$ is the usual \define{$h^*$-polynomial} $h^*(P;u) \in \Z[u]$ of $P$ introduced by Stanley in   \cite{StaDecompositions}.
We define the \define{local  weighted $h^*$-polynomial} $\lc(P,\nu;u) \in \Z[\Q/\Z][u]$ by
\[
\lc(P,\nu;u) = \sum_{Q \subseteq P} (-1)^{\dim P - \dim Q}h^*(Q,\nu|_Q;u)g([Q,P]^*;u).
\]
Here the sum runs over all faces of $P$ including the empty face. When $P$ is the empty polytope, we set $l^*(P,\nu;u) = 1$.
As in Remark~\ref{r:forget}, the corresponding polynomial with coefficients in $\Z$ is the usual \define{local $h^*$-polynomial} $\lc(P;u) \in \Z[u]$ of $P$,
 introduced by Stanley in \cite[Example~7.13]{StaSubdivisions}, and later independently by Borisov and Mavlyutov in \cite{BMString}.
By Theorem~\ref{t:Eulerianinverse} and Remark~\ref{r:non-negative},
the  weighted $h^*$-polynomial can be recovered as a `non-negative' combination of local weighted  $h^*$-polynomials:
\[
h^*(P,\nu;u) = \sum_{Q \subseteq P} \lc(Q,\nu|_Q;u)g([Q,P];u).
\]

We present some properties and examples below. For the corresponding invariants over $\Z$, proofs of the properties below can be found in   \cite{BNCombinatorial} and \cite{Combinatorics}.

\begin{example}\label{e:volume}
One may verify that \eqref{e:leading} translates into the following:
\[
h^*(P,\nu;1) = \sum_{ \substack{F \in \cS \\ \dim F = \dim P} }  \frac{\Vol(F)}{m_F} \sum_{i = 0}^{m_F - 1} [i/m_F] \in \Z[\Q/\Z],
\]
where $m_F$ is the minimal positive integer such that  $m_F \nu|_{F}$ is an affine 
function with respect to $M$.
\end{example}

\begin{example}\label{e:simplex}
Suppose that $P$ is a simplex and $\nu$ is $\Q$-affine, and let $v_0, \ldots, v_n$ denote the primitive integer vectors on the rays of $C$ corresponding to the vertices of $P$.
Let $\bx = \{ v \in C \cap (M \times \Z) \mid v = \sum_{i = 0}^n a_i v_i, 0 \le a_i < 1 \}$.
If
$\pr: N_\R \times \R \rightarrow \R$ denotes projection onto the second coordinate,
then
\[
h^*(P,\nu;u) = \sum_{v \in \bx} w(v) u^{\pr(v)} \in \Z[\Q/\Z][u],
\]
\[
\lc(P,\nu;u) = \sum_{v \in \bx \cap \Int(C)} w(v) u^{\pr(v)} \in \Z[\Q/\Z][u].
\]
\end{example}

\begin{remark}\label{r:symmetry}
One may verify that weighted Ehrhart reciprocity is equivalent to the following statements about the weighted $h^*$-polynomial and local weighted $h^*$-polynomial respectively
(cf. \cite[Lemma~7.11]{Combinatorics}).
Firstly, the weighted $h^*$-polynomial is \define{acceptable} in the sense of Stanley \cite{StaSubdivisions}. That is,
\[
u^{\dim P + 1} \bar{h^*(P,\nu;u^{-1})}   
= \sum_{Q \subseteq P} h^*(Q,\nu|_Q;u)(u - 1)^{\dim P - \dim Q},
\]
where the above sum ranges over all (possibly empty) faces $Q$ of $P$. Secondly, the local weighted $h^*$-polynomial is symmetric in the sense that
\[
\lc(P,\nu;u) = u^{\dim P + 1} \bar{\lc(P,\nu;u^{-1})}.
\]
\end{remark}

\begin{example}\label{e:weightedlocal}
If  $P$ is non-empty, then
$h^*(P,\nu;u)$ and $\lc(P,\nu;u)$ have degree in $u$ at most $\dim P$, with the coefficient of $u^{\dim P}$ in both cases equal to
$\sum_{ v \in \Int(P) \cap M } \bar{w(v)}$.
By the symmetry in Remark~\ref{r:symmetry} above, this implies that $\lc(P,\nu;u)$ has constant term $0$ and linear coefficient $\sum_{ v \in \Int(P) \cap M } w(v)$.
The polynomial $h^*(P,\nu;u)$ has constant term $1$ and linear coefficient $\sum_{v \in P \cap M} w(v) - \dim P - 1$.

\end{example}

\begin{remark}\label{r:decompose}
Let $\cS$ be a lattice polyhedral subdivision of $P$ that refines $\cS(\nu)$ such that $\nu$ takes integer values on the vertices of $\cS$. 
Then it follows from the definitions that for every positive integer $m$,
\[
f(P,\nu;m) = \sum_{ F \in \cS } \sum_{v \in \Int(mF) \cap M} w(v).
\]
Using weighted Ehrhart reciprocity and following the proof of \cite[Lemma~7.12]{Combinatorics}, the above statement is equivalent to any of the three statements below:
\[
h^*(P,\nu;u) = \sum_{ \substack{ F \in \cS \\ \sigma(F) = P }}  h^*(F,\nu|_F;u) (u - 1)^{\dim P - \dim F},
\]
\[
h^*(P,\nu;u) = \sum_{ F \in \cS }  \lc(F,\nu|_F;u)  h(\lk_\cS(F);u),
\]
\[
\lc(P,\nu;u) = \sum_{ F \in \cS }  \lc(F,\nu|_F;u) l_P(\cS,F;u).
\]


\end{remark}

\begin{remark}\label{r:non-negative2}
One may always choose a lattice triangulation $\cS$ of $P$ that refines $\cS(\nu)$ such that $\nu$ takes integer values on the vertices of $\cS$. In fact, there exists a choice with the same vertices as $\cS(\nu)$ (see, for example, \cite{OPLinear}).
Then Remark~\ref{r:decompose} together with Example~\ref{e:simplex} give  explicit formulas that one may use to compute $h^*(P,\nu;u)$ and $\lc(P,\nu;u)$ in practice. In particular, using
Remark~\ref{r:non-negative},
it follows that all integer coefficients in  $h^*(P,\nu;u)$ and $\lc(P,\nu;u)$ are non-negative.
\end{remark}

\begin{remark}\label{r:lower}
The following lower bound theorem generalizes \cite[Theorem~7.20]{Combinatorics} in the case when $\nu \equiv 0$: if we write $\lc(P,\nu;u) = \sum_{k \in \Q/\Z} \lc(P,\nu;u)_{[k]} [k]$ and
$\lc(P,\nu;u)_{[0]} = \lc_{1}(P,\nu)_{[0]}u + \cdots + l_{\dim P}^*(P,\nu)_{[0]} u^{\dim P}$, then
\[
\lc_{1}(P,\nu)_{[0]} =  \# \{ v \in \Int(P) \cap \Z^d \mid w(v) = 0 \} \le \lc_{i}(P,\nu)_{[0]}
\]
for $1 \le i \le \dim P$. This follows by the proof of \cite[Theorem~7.20]{Combinatorics}. Explicitly,
consider a regular, lattice polyhedral subdivision $\cS$ of $P$ that refines $\cS(\nu)$ such that the lattice points of weight $0$ in $P$ are precisely the vertices of $\cS$. Then $\nu$ takes integer values on the vertices of $\cS$, and every positive-dimensional  cell in $\cS$ contains no interior lattice points of weight $0$.
By Remark~\ref{r:decompose} and Example~\ref{e:weightedlocal},
\[
\lc(P,\nu;u)_{[0]} = l_P(\cS;u) + \beta(u)u^2,
\]
where $l_P(\cS;u)$ has non-negative, symmetric, unimodal coefficients and the polynomial $\beta(u)$ has non-negative integer coefficients. The result follows.
\end{remark}

\begin{example}\label{e:old}
The following example appeared in \cite[Example~1.1]{Weighted}. Let  $M = \mathbb{Z}^{2}$ and let $P$ be the lattice polytope with vertices $(1,0)$,$(0,2)$,
$(-1,2)$, $(-2,1)$, $(-2,0)$ and $(0,-1)$. Let $\nu: P \rightarrow \R$ be the piecewise $\Q$-linear function satisfying $\nu(0) = 0$ and $\nu(v)  = 1$ for every vertex $v$ of $P$.
We write $f(P,\nu; m) = \sum_{[k] \in \Q/\Z}  f_{[k]}(P,\nu; m)[k]$ below.
Then the invariants above may be computed to be:


\[ h^*(P,\nu; u)  = 1+ 4u + u^{2}    + 2u(1 + u)[1/2] + u[2/3]  + u^2[1/3]      \]
\[ h^*(P; u)  =  1+ 7u + 4u^{2}   \]

\[ l^*(P,\nu; u)  = u(1 + u)(1 + 2[1/2]) + u[2/3] + u^2[1/3]    \]
\[ l^*(P; u)  = 4u(1 + u)  \]

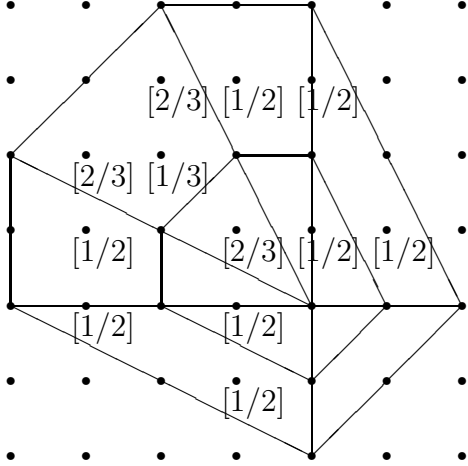
\setlength{\unitlength}{1cm}
\begin{picture}(12,6.5)

\put(7,5){$f_{[0]}(P,\nu; m) = 3m^{2} + 3m + 1$}
\put(7,4){$f_{[1/2]}(P,\nu; m) = 2m^{2}$}
\put(7,3){$f_{[1/3]}(P,\nu; m) = m(m - 1)/2$}
\put(7,2){$f_{[2/3]}(P,\nu; m) = m(m + 1)/2$}
\put(7,1){$f(P; m) = 6m^{2} + 3m + 1$}

\put(0,0){\circle*{0.1}}
\put(0,1){\circle*{0.1}}
\put(0,2){\circle*{0.1}}
\put(0,3){\circle*{0.1}}
\put(0,4){\circle*{0.1}}
\put(0,5){\circle*{0.1}}
\put(0,6){\circle*{0.1}}
\put(1,0){\circle*{0.1}}
\put(1,1){\circle*{0.1}}
\put(1,2){\circle*{0.1}}
\put(1,3){\circle*{0.1}}
\put(1,4){\circle*{0.1}}
\put(1,5){\circle*{0.1}}
\put(1,6){\circle*{0.1}}
\put(2,0){\circle*{0.1}}
\put(2,1){\circle*{0.1}}
\put(2,2){\circle*{0.1}}
\put(2,3){\circle*{0.1}}
\put(2,4){\circle*{0.1}}
\put(2,5){\circle*{0.1}}
\put(2,6){\circle*{0.1}}
\put(3,0){\circle*{0.1}}
\put(3,1){\circle*{0.1}}
\put(3,2){\circle*{0.1}}
\put(3,3){\circle*{0.1}}
\put(3,4){\circle*{0.1}}
\put(3,5){\circle*{0.1}}
\put(3,6){\circle*{0.1}}
\put(4,0){\circle*{0.1}}
\put(4,1){\circle*{0.1}}
\put(4,2){\circle*{0.1}}
\put(4,3){\circle*{0.1}}
\put(4,4){\circle*{0.1}}
\put(4,5){\circle*{0.1}}
\put(4,6){\circle*{0.1}}
\put(5,0){\circle*{0.1}}
\put(5,1){\circle*{0.1}}
\put(5,2){\circle*{0.1}}
\put(5,3){\circle*{0.1}}
\put(5,4){\circle*{0.1}}
\put(5,5){\circle*{0.1}}
\put(5,6){\circle*{0.1}}
\put(6,0){\circle*{0.1}}
\put(6,1){\circle*{0.1}}
\put(6,2){\circle*{0.1}}
\put(6,3){\circle*{0.1}}
\put(6,4){\circle*{0.1}}
\put(6,5){\circle*{0.1}}
\put(6,6){\circle*{0.1}}

\put(3.8,2.6){$[1/2]$}
\put(3.8,4.6){$[1/2]$}
\put(4.8,2.6){$[1/2]$}
\put(2.8,1.6){$[1/2]$}
\put(0.8,1.6){$[1/2]$}
\put(0.8,2.6){$[1/2]$}
\put(0.8,3.6){$[2/3]$}
\put(1.8,3.6){$[1/3]$}
\put(1.8,4.6){$[2/3]$}
\put(2.8,2.6){$[2/3]$}
\put(2.8,4.6){$[1/2]$}
\put(2.8,0.6){$[1/2]$}


\linethickness{0.075mm}
\put(4,0){\line(0,1){6}}
\put(0,2){\line(1,0){6}}
\put(4,2){\line(-1,2){2}}
\put(4,2){\line(-2,1){4}}

\linethickness{0.2mm}
\put(3,4){\line(1,0){1}}
\put(2,2){\line(0,1){1}}
\put(2,6){\line(1,0){2}}
\put(0,2){\line(0,1){2}}
\put(4,4){\line(1,-2){1}}
\put(5,2){\line(-1,-1){1}}
\put(2,2){\line(2,-1){2}}
\put(2,3){\line(1,1){1}}
\put(4,6){\line(1,-2){2}}
\put(6,2){\line(-1,-1){2}}
\put(0,2){\line(2,-1){4}}
\put(0,4){\line(1,1){2}}

\end{picture}

\end{example}

\subsection{Mixed invariants}\label{s:mixed} We introduce our main combinatorial invariants below.
As in Remark~\ref{r:forget}, the corresponding invariants with coefficients in $\Z$ first appeared in \cite{Combinatorics}.
A geometric description of these invariants is provided in Corollary~\ref{c:explicit}.

\begin{definition}\label{d:new}
Let $P$ be a lattice polytope and let $\nu: P \rightarrow \R$ be the convex graph of an integral height function on $P$,
with corresponding regular lattice polyhedral subdivision $\cS(\nu)$ of $P$.
Then the \define{weighted limit mixed $h^*$-polynomial} $h^*(P,\nu;u,v) \in \Z[\Q/\Z][u,v]$ of $(P,\nu)$ is
\[
h^*(P,\nu;u,v) := \sum_{F \in \cS(\nu)} v^{\dim F + 1}\lc(F, \nu|_F; uv^{-1})  h(\lk_{\cS(\nu)}(F);uv).
\]
The \define{local weighted  limit mixed $h^*$-polynomial} $\lc(P,\nu;u,v) \in \Z[\Q/\Z][u,v]$ of $(P,\nu)$ is
\[
\lc(P, \nu; u,v) :=   \sum_{F \in \cS(\nu)} v^{\dim F + 1}\lc(F, \nu|_F; uv^{-1})  l_P(\cS(\nu),F;uv). 
\]
The \define{weighted refined limit mixed $h^*$-polynomial} $h^*(P,\nu;u,v,w) \in \Z[\Q/\Z][u,v,w]$ of $(P,\nu)$ is
\[
h^*(P,\nu;u,v,w) = \sum_{Q \subseteq P} w^{\dim Q + 1}\lc(Q, \nu|_Q;u,v) g([Q,P];uvw^2).
\]
If $P$ is empty, then $h^*(P,\nu;u,v,w) = h^*(P,\nu;u,v) =  \lc(P, \nu;u,v)  = 1$.  We write
\[
h^*(P,\nu;u,v,w) = 1 + uvw^2 \cdot \sum_{0 \le p,q,r \le \dim P  - 1} h^*_{p,q,r}(P,\nu) u^{p}v^{q}w^r,
\]
where $h^*_{p,q,r}(P,\nu) = \sum_{[k] \in \Q/\Z} h^*_{p,q,r}(P,\nu)_{[k]} [k] \in \Z[\Q/\Z]$, for some $h^*_{p,q,r}(P,\nu)_{[k]} \in \Z$.
\end{definition}






We summarize the main properties of these invariants in the following theorem. 
They may all be deduced from the definitions and the properties in the previous sections, and hence we omit the proof.
We refer the reader to \cite[Theorem~9.2, Theorem~9.3,Theorem~9.4]{Combinatorics} for a similar statement for the corresponding  invariants with coefficients in $\Z$.

\begin{theorem}\label{t:combinatorics}
Let $P$ be a lattice polytope and let $\nu: P \rightarrow \R$ be the convex graph of an integral height function on $P$,
with corresponding regular lattice polyhedral subdivision $\cS(\nu)$ of $P$. Then the following holds:

\begin{enumerate}

\item\label{prop'}  The weighted refined limit mixed $h^*$-polynomial satisfies the following symmetries:
\[
h^*(P,\nu;u,v,w) = \bar{h^*(P,\nu;v,u,w)},
\]
\[
h^*(P,\nu;u,v,w) = \bar{h^*(P,\nu;u^{-1},v^{-1},uvw)}.
\]

\item\label{prop1} The  weighted refined limit mixed $h^*$-polynomial specializes to the weighted  limit mixed $h^*$-polynomial:
\[
h^*(P,\nu;u,v,1) = h^*(P,\nu;u,v).
\]


\item\label{prop3} The weighted refined limit mixed $h^*$-polynomial specializes to the weighted $h^*$-polynomial:
\[
h^*(P,\nu;u,1,1) = h^*(P,\nu;u,1) = h^*(P,\nu;u).
\]

\item\label{prop3'} The local weighted limit mixed $h^*$-polynomial specializes to the local weighted $h^*$-polynomial:
\[
\lc(P,\nu;u,1) = \lc(P,\nu;u).
\]

\item\label{prop4} The degree of $h^*(P,\nu;u,v,w)$ as a polynomial in $w$ is at most $\dim P + 1$.  Moreover, the coefficient of
$w^{\dim P + 1}$ is the local weighted limit mixed $h^*$-polynomial
$\lc(P, \nu; u,v)$. We have symmetries:
\[
\lc(P,\nu;u,v) = \bar{\lc(P,\nu;v,u)} = (uv)^{\dim P + 1} \bar{\lc(P,\nu;u^{-1},v^{-1})}.
\]

\item\label{prop111} We have the following:
\[
h^*(P, \nu;u,v,w) =  \sum_{Q \subseteq P}w^{\dim Q + 1} \lc(Q, \nu|_Q; u,v)   g([Q,P]; uvw^2),
\]
\[
w^{\dim P + 1} \lc(P, \nu; u,v) =  \sum_{Q \subseteq P} (-1)^{\dim P - \dim Q}h^*(Q, \nu|_Q;u,v,w) g([Q,P]^*;uvw^2).
\]

\item\label{prop6} 
We have the following:
\[
h^*(P, \nu;u,v) = \sum_{ \substack{ F \in \cS(\nu) \\ \sigma(F)= P } } h^*(F,\nu|_F;u,v) (uv - 1)^{\dim P - \dim F}.
\]

\item\label{prop7}
If $\nu$ is $\Q$-affine, then $\lc(P, \nu; u,v) =   v^{\dim P + 1}\lc(P, \nu; uv^{-1})$ and
\[
h^*(P,\nu;u,v,w)  = h^*(P,\nu;uw,vw).
\]

\item\label{prop8}
The  coefficients $h^*_{p,q,r}(P,\nu)_{[k]}$ are  non-negative integers.
For $0 \le j \le r$ and $[k] \in \Q/\Z$, the sequence
$\{ h^{i+ j,i,r}(P,\nu)_{[k]} \mid 0 \le i \le r -j \}$ is symmetric and unimodal.
Moreover, $h^*_{j,0,r}(P,\nu)_{[0]} \le h^*_{j-i,i,r}(P,\nu)_{[0]}$ for $0 \le i \le j$ and  $h^*_{r,j,r}(P,\nu)_{[0]} \le h^*_{r-i, j+i,r}(P,\nu)_{[0]}$
for $0 \le i \le r -j$.

\end{enumerate}

\end{theorem}

\begin{remark}
As in \cite[Section~9]{Combinatorics}, after fixing $r$, one may view the coefficients $h^*_{p,q,r}(P,\nu)$ in a diamond, for example,
by placing $h^*_{p,q,r}(P,\nu)$ at point $(q - p, p + q)$ in $\Z^2$.
The resulting  diamond of coefficients  is symmetric  with respect to its symmetry about the horizontal axis and symmetric up to conjugation with respect to its symmetry about the vertical axis.
Fixing $[k] \in \Q/\Z$, the coefficients of $[k]$ in each vertical strip of the diamond form a symmetric, unimodal sequence. The coefficients of $[0]$ in each horizontal strip of the diamond satisfy
the following lower bound theorem: the first entry  is a lower bound for the other entries.
\end{remark}

\begin{example}\label{e:simplex2}
As in Example~\ref{e:simplex}, suppose that $P$ is a simplex and $\nu$ is $\Q$-affine. 
For every face $Q$ of $P$, let $C_Q$ denote the cone over $Q \times \{1 \}$  in $M_\R \times \R$, and set $C = C_P$. Note that if $Q$ is empty, then $C_Q$ is the origin.
Let
$v_0, \ldots, v_n$ denote the primitive integer vectors on the rays of $C$ corresponding to the vertices of $P$.
Let $\bx = \{ v \in C \cap (M \times \Z) \mid v = \sum_{i = 0}^n a_i v_i, 0 \le a_i < 1 \}$.
If
$\pr: N_\R \times \R \rightarrow \R$ denotes projection onto the second coordinate,
then
\[
h^*(P,\nu;u,v,w) = \sum_{Q \subseteq P} \sum_{v \in \bx \cap \Int(C_Q)} w(v) u^{\pr(v)} v^{\dim Q + 1 - \pr(v)} w^{\dim Q + 1} \in \Z[\Q/\Z][u,v,w].
\]
\[
\lc(P,\nu;u,v) =  \sum_{v \in \bx \cap \Int(C)} w(v) u^{\pr(v)} v^{\dim P + 1 - \pr(v)}  \in \Z[\Q/\Z][u,v].
\]
\end{example}



We present explicit descriptions of some of the coefficients below.

\begin{example}\label{e:smallterms}
Let $P$ be a lattice polytope and let $\nu: P \rightarrow \R$ be the convex graph of an integral height function on $P$,
with corresponding regular lattice polyhedral subdivision $\cS(\nu)$ of $P$. 
Then for $q,r > 0$, 
\[
h^*_{0,q,r}(P,\nu) = h^*_{r - q,r,r}(P,\nu) = \bar{h^*_{q,0,r}(P,\nu)} =  \bar{h^*_{r,r - q,r}(P,\nu)}  =   \sum_{\substack{F \in \cS(\nu) \\\ \dim F = q + 1 \\\ \dim \sigma(F) = r + 1}} \sum_{v \in \Int(F) \cap M} w(v),
\]
\[
h^*_{0,0,r}(P,\nu) = h^*_{r,r,r}(P,\nu) =  \sum_{\substack{F \in \cS(\nu) \\\ \dim F \le 1 \\\ \dim \sigma(F) = r + 1}} \sum_{v \in \Int(F) \cap M} w(v),
\]
\[
\dim P + 1 + h^*_{0,0,0}(P,\nu) =  \sum_{\substack{Q \subseteq P \\\ \dim Q \le 1}} \sum_{v \in \Int(Q) \cap M} w(v).
\]
When $\dim P = 2$, this gives an explicit description of $h^*(P,\nu;u,v,w)$:
\[
h^*(P;\nu;u,v,w) = 1 + uvw^2  \big[ h^*_{0,0,0}(P,\nu) + w\big[(1 + uv)h^*_{0,0,1}(P,\nu) + vh^*_{0,1,1}(P,\nu) + u\bar{h^*_{0,1,1}(P,\nu)}\big]\big].
\]
When $\dim P = 3$, this gives an explicit description of all terms in $h^*(P,\nu;u,v,w)$ except $h^*_{1,1,2}$. The remaining term can be deduced from
the expression for $h^*(P,\nu;1,1,1) = h^*(P,\nu;1)$ in Example~\ref{e:volume}.
\end{example}

\begin{example}
In the case of Example~\ref{e:old}, one computes:
\[
h^*(P;\nu;u,v,w) = 1 + uvw^2  \big[ 3 + w\big[(1 + uv)(1 + 2[1/2]) + v[2/3] + u[1/3]\big]\big],
\]
\[
h^*(P;\nu;u,v) = 1 + uv  \big[ 4 + uv  + 2[1/2](1 + uv) + v[2/3] + u[1/3] \big],
\]
\[
\lc(P;\nu;u,v) =  uv  \big[ (1 + uv)(1 + 2[1/2]) + v[2/3] + u[1/3] \big].
\]
\end{example}



\section{Degenerations of  hypersurfaces}\label{s:hyper}

In this section, we use Theorem~\ref{t:comp} together with the combinatorics of Section~\ref{s:weighted} to prove a formula for the equivariant refined limit Hodge-Deligne polynomial of a sch\"on hypersurface
of a torus. We also give an equivalent formula for the equivariant refined limit Hodge-Deligne polynomial associated to the intersection cohomology of a canonical compactification of such a hypersurface.
The corresponding results for invariants with coefficients in $\Z$ appear in \cite{Geometry}. We will continue with the notation of Section~\ref{s:degenerations} and Section~\ref{s:weighted}.

Let $M$ be a lattice and let $N = \Hom(M,\Z)$.
Let $X^\circ =  \{ f = 0 \} \subseteq \Spec \K[M] \cong (\K^*)^n$ be a sch\"on hypersurface, where $f = \sum_{u \in M} \gamma_u(t) x^u$ for some $\gamma_u(t) \in \C(t)$.
The \define{Newton polytope} $P$ of $X^\circ$ is the convex hull of $S = \{ u \in M \mid \gamma_u(t) \ne 0 \}$ in $M_\R$.
 Note that  $P$ may be viewed as
 a full-dimensional lattice polytope in its affine span 
 with induced lattice structure, and $X^\circ \cong X' \times (\K^*)^k$,
 for some $k$, where $X'$ is a sch\"on hypersurface in a torus of dimension $\dim P$ with Newton polytope $P$.
 Hence we may and will assume that $\dim P = n$.

 Consider the induced integral height function  $\omega_f: S \rightarrow \Z$ defined by $\omega_f(u) = \ord_t \gamma_u(t)$.
Recall from Section~\ref{subdivisions} that we may consider the convex graph $\nu_f$ associated to $\omega_f$, with corresponding
regular lattice polyhedral subdivision $\cS(\nu_f)$.

\begin{remark}\label{r:existence}
Conversely, assume that $\nu: P \rightarrow \R$ is the convex graph of an integral height function on $P$. Since sch\"onness is a generic condition, it follows that there exists a
sch\"on hypersurface $X^\circ =  \{ f = 0 \} \subseteq \Spec \K[M]$ such that $\nu_f = \nu$ \cite[Remark~5.1]{Geometry}.
\end{remark}

Tropical geometry of hypersurfaces reduces to the study of Newton polytopes and polyhedral subdivisions \cite{GKZ,RSTFirst}.
In particular, $\Trop(X^\circ)$ admits a polyhedral structure $\Sigma$ with cells $\gamma$ in inclusion-reversing correspondence with positive-dimensional cells $F$ of $\cS(\nu_f)$, such that the bounded
cells of $\Sigma$ correspond to the cells of $\cS(\nu_f)$ not contained in the boundary of $P$. For example, if $F$ is a maximal dimensional cell in $\cS(\nu_f)$, then $\nu_f|_F$ is a $\Q$-affine function
and the corresponding $\Q$-linear function is an rational point  in $N_\Q$, also denoted $\nu_f|_F$. Then $\gamma = -\nu_f|_F$ is a vertex in $\Sigma$.

\begin{example}\label{e:simple3}
Following on with Example~\ref{e:simple} and Example~\ref{e:simple2}, let $X^\circ = \{ \sum_{i  = 1}^n 
x_i^{m_i} = t^e \} \subseteq (\K^*)^n$, for some  
$m_i \in \Z_{>0}$ and $e \in \Z$.
Let $f_1, \ldots, f_n$ denote
the standard basis vectors of $M_R \cong \R^n$. Then $P$ is the convex hull of the origin and $\{ m_i f_i \mid 1 \le i \le n \}$, and $\nu_f$ is the $\Q$-affine function corresponding to the $\Q$-linear function
$(-e/m_1, \ldots, -e/m_n) \in N_\Q$. As we have seen, $\Trop(X^\circ)$ has the structure of  a fan with
vertex $(e/m_1, \ldots, e/m_n) \in N_\Q$.
\end{example}


Write $f = \sum_{u \in M}  \lambda_{u}  t^{\omega_f(u)} g_u(t) x^{u}$ for some  $\lambda_{u} \in \C$ and
$g_u(t) \in \C(t)$ such that $g(1) = 1$.
Let $F$ be a face of $\cS(\nu_f)$ corresponding to a cell  $\gamma$ in $\cS(\nu_f)$. Then
\begin{equation}\label{e:init}
\init_\gamma X^\circ =  \{ \sum_{ \substack{ u \in F \cap M \\ \omega_f(u) = \nu_f(u)} } \lambda_u x^u = 0 \} \subseteq (\C^*)^n.
\end{equation}
Let $M_F$ denote the intersection of $M$ with the translation of
the affine span of $F$ to the origin. Then 
equation in the right hand side of \eqref{e:init} determines a sch\"on complex hypersurface $V_F^\circ \in \Spec \C[M_F]$ and
$\nu_f|_F$ determines a $\Q$-linear function on $(M_F)_\Q$. Multiplication by $\exp(-2\pi \sqrt{-1} \nu_f|_F)$ determines a good $\hat{\mu}$-action on
$V_F^\circ$. In this case, Theorem~\ref{t:comp} translates into the following corollary.

\begin{corollary} \label{c:hyper}
Let $X^\circ \subseteq (\K^*)^n$ be a sch\"{o}n hypersurface, with associated Newton polytope and convex graph of an integral height function $(P,\nu)$ and $\dim P = n$.
Then the  motivic nearby fiber of $X^\circ$ is given by
\[
\psi_{X^\circ} =  \sum_{\substack{F \in \cS(\nu) \\\ \sigma(F) = P}}  [V_{F}^\circ \circlearrowleft \hat{\mu}](1 - \L)^{\dim P - \dim F}.
\]
Here $\sigma(F)$ denotes the smallest face of $P$ containing $F$, and $V_{F}^\circ$ is a sch\"on complex hypersurface of a torus with  good $\hat{\mu}$-action
induced by multiplication by $\exp(-2\pi \sqrt{-1} \nu|_F)$.
\end{corollary}

\begin{remark}
In fact, the sum in Corollary~\ref{c:hyper}  runs only over the positive dimensional faces  in $\cS(\nu)$, since  when $F$ is a vertex of $\cS(\nu)$,  $V_{F}^\circ = \emptyset$.
Here we follow the convention that $[V] = 0 \in  K_0^{\hat{\mu} }(\Var_\C)$ if $V$ is empty.
\end{remark}

\begin{example}\label{e:affine}
Let $X^\circ \subseteq (\K^*)^n$ be a sch\"{o}n hypersurface, with associated Newton polytope and convex graph of an integral height function $(P,\nu)$ and $\dim P = n$.
Assume that $\nu$ is $\Q$-affine. Then the associated polyhedral subdivision  $\cS(\nu)$ of $P$ is trivial and Corollary~\ref{c:hyper} states that $\psi_{X^\circ} = [V^\circ_P \circlearrowleft \hat{\mu}]$.
\end{example}

\begin{example}\label{e:simpleaffine}
Let $V^\circ =  \{ \sum_{u \in P \cap M} \lambda_u x^u = 0 \} \subseteq (\C^*)^n$ be a sch\"on complex hypersurface with Newton polytope $P$ and  $\dim P = n$.  Let $\nu$ be the convex graph of an integral height function on $P$,
and assume that $\nu$ is $\Q$-affine   
and $\nu(u) \in \Z$ whenever $\lambda_u \ne 0$.
Then
$V^\circ$ is invariant under multiplication by the associated element $\exp(-2\pi \sqrt{-1} \nu) \in (\C^*)^n$, and we have an induced good $\hat{\mu}$-action on $V^\circ$. For example,
the hypersurfaces $V_{F}^\circ$ with corresponding good $\hat{\mu}$-action in the statement of Corollary~\ref{c:hyper} all appear in this way. 

In this case, the variety $X^\circ = \{ f = \sum_{u \in P \cap M} \lambda_u t^{\nu(u)} x^u = 0 \} \subseteq (\K^*)^n$ is sch\"on,   and may be viewed as a family $f: X^\circ \rightarrow \C^*$ with
isomorphic fibers. Explicitly, we may identify $f^{-1}(1)$ with $V^\circ$,  and
multiplication by $\exp(-2\pi \sqrt{-1} \lambda \nu) \in (\C^*)^n$, for any $\lambda \in \C$, gives an isomorphism between $f^{-1}(1)$ and $f^{-1}(\exp(2\pi \sqrt{-1} \lambda))$.
Considering  $0 \le \lambda \le 1$, we see that we may identify the good $\hat{\mu}$-action on $V^\circ$ above with the good $\hat{\mu}$-action on $f^{-1}(1)$ induced by monodromy.
Note that since the monodromy action is finite, the corresponding monodromy operator $T$ on cohomology is semi-simple. As in Example~\ref{e:affine},
Corollary~\ref{c:hyper} states that $\psi_{X^\circ} = [V^\circ \circlearrowleft \hat{\mu}]$.
\end{example}

Recall the definition of the weighted refined limit mixed $h^*$-polynomial $h^*(P,\nu;u,v,w) \in \Z[\Q/\Z][u,v,w]$ from Definition~\ref{d:new}.
Our goal is to prove the following formula for the equivariant refined limit Hodge-Deligne polynomial.

\begin{theorem}\label{t:mainhyper}
Let $X^\circ \subseteq (\K^*)^n$ be a sch\"{o}n hypersurface, with associated Newton polytope and convex graph of an integral height function $(P,\nu)$ and $\dim P = n$. Then the equivariant  refined limit Hodge-Deligne polynomial
$E(X_\infty^\circ, \hat{\mu} ;u,v,w) \in \Z[\Q/\Z][u,v,w]$ of $X^\circ$ is given by
\begin{equation*}
uvw^2E(X_\infty^\circ, \hat{\mu};u,v,w) = (uvw^2 - 1)^{\dim P} + (-1)^{\dim P + 1} h^*(P,\nu;u,v,w).
\end{equation*}
\end{theorem}

We will also prove an equivalent statement involving intersection cohomology. We will use middle-perversity throughout.   As explained in detail in \cite[Section~6]{Geometry}, given a projective variety $X$ over $\K$, for each of the geometric invariants described in
 Section~\ref{s:invariants}, one may consider the corresponding invariant with respect to intersection cohomology rather than cohomology 
 e.g. the equivariant  refined limit Hodge-Deligne polynomial $E_{\inter}(X_\infty, \hat{\mu};u,v,w)$ associated to the  intersection cohomology of $X$.  The key benefit of using intersection cohomology is the fact that the intersection cohomology groups of
 a complex projective variety admit a pure (rather than just mixed) Hodge structure.

We will need the  following sum-over-strata formula. Let $X^\circ \subseteq (\K^*)^n$ be a sch\"{o}n hypersurface, with associated Newton polytope and convex graph of an integral height function $(P,\nu)$ and $\dim P = n$.
 Let $\Delta_P$ denote the normal fan to $P$ with cones $\sigma$ in inclusion-reserving bijection with the faces $Q_{\sigma}$ of $P$. Let $X$ denote the closure of $X^\circ$ in the projective toric variety
 $\P(\Delta_P)_\K$ over $\K$ corresponding to $\Delta_P$. Then $X$ admits a stratification
 $X = \bigcup_{\sigma\in\Delta_P} X_\sigma^\circ$, where $X_\sigma^\circ$ is a sch\"on hypersurface with Newton polytope $Q_{\sigma}$ and  convex graph of an integral height function $\nu|_{Q_\sigma}$.
 The theorem below is proved in the non-equivariant setting in \cite[Theorem~6.1]{Geometry}, but the proof goes through unchanged in our situation.


\begin{theorem}\cite[Theorem~6.1]{Geometry} \label{t:sumoverstrata}
Let $X^\circ \subseteq (\K^*)^n$ be a sch\"{o}n hypersurface, with associated Newton polytope and convex graph of an integral height function $(P,\nu)$  and $\dim P = n$.
Let $X$ denote the closure of $X^\circ$ in the projective toric variety over $\K$ corresponding to the normal fan of $P$. Then
the equivariant refined limit Hodge-Deligne polynomial obeys
\[E_{\inter}(X_\infty, \hat{\mu};u,v,w)=\sum_{\sigma\in\Delta_P} E((X_\sigma^\circ)_{\infty}, \hat{\mu};u,v,w)g([0,\sigma];uvw^2).
\]
\end{theorem}

Our goal is to prove the following theorem.

 \begin{theorem}\label{t:mainintersection}
Let $X^\circ \subseteq (\K^*)^n$ be a sch\"{o}n hypersurface, with associated Newton polytope and convex graph of an integral height function $(P,\nu)$  and $\dim P = n$.
Let $X$ denote the closure of $X^\circ$ in the projective toric variety over $\K$ corresponding to the normal fan of $P$.
Then the equivariant  refined limit Hodge-Deligne polynomial $E_{\inter}(X_\infty, \hat{\mu};u,v,w)  \in \Z[\Q/\Z][u,v,w]$ associated to the intersection cohomology of $X$ is given by
\[
uvw^2E_{\inter}(X_\infty, \hat{\mu};u,v,w) = uvw^2E_{\inter,\Lef}(P;uvw^2) + (-1)^{\dim P + 1}w^{\dim P + 1} \lc(P,\nu;u,v),
\]
where
\begin{equation}\label{e:interlef}
(t - 1)E_{\inter,\Lef}(P;t) = t^{\dim P}g([\emptyset,P]^*;t^{-1})- g([\emptyset,P]^*;t)
\end{equation}
is defined in terms of Stanley's g-polynomial (see Definition~\ref{d:g}).
\end{theorem}

The following lemma is a formal consequence of Theorem~\ref{t:sumoverstrata}. Details are provided in \cite[Lemma~6.2]{Geometry}.

\begin{lemma}\label{l:equiv}
Theorem~\ref{t:mainhyper} and Theorem~\ref{t:mainintersection} are equivalent.
\end{lemma}

It is well known that all interesting cohomology above occurs in middle dimension. We review the relevant facts below and refer the reader to \cite[Sections~5.3,6]{Geometry} for details.
We have weak Lefschetz theorems stating that the Gysin maps
$H^{m}_c(X^\circ_{\gen})\rightarrow H^{m+2}_c((\C^*)^n)$ and $IH^{m}(X_{\gen})\rightarrow IH^{m+2}(\P(\Delta_P)_\C)$ are surjective for $m\geq n - 1$ and an isomorphism if $m > n - 1$.
One may deduce from Theorem~\ref{t:sumoverstrata} that the monodromy maps associated to the cohomology with compacts supports and intersection cohomology of $(\K^*)^n$ and
$\P(\Delta_P)_\K$ respectively are trivial. Since $X^\circ_{\gen}$ is affine, $H^{m}_c(X^\circ_{\gen}) = 0$ for $m < n - 1$. Also, Poincar\'{e} duality for intersection cohomology means that
the Gysin isomorphism above determines $IH^{m}(X_{\gen})$ for $m < n - 1$.
We let $H^{n - 1}_{\prim,c}(X^\circ_{\infty})$ and $IH^{n - 1}_{\prim}(X_{\infty})$ denote the kernels of the above Gysin maps with induced filtrations.
These groups constitute the `interesting' cohomology. Consider the contributions of these groups (up to a sign) to the equivariant refined limit  Hodge-Deligne polynomial:
\[
E_{\prim}(X_\infty^\circ, \hat{\mu};u,v,w) := 
\sum_{p,q,r} \sum_{\alpha \in \Q/\Z} h^{p,q,r}(H^{n - 1}_{\prim,c}(X^\circ_{\infty}))_\alpha \alpha u^p v^q w^r,
\]
\[
E_{\inter, \prim}(X_\infty, \hat{\mu};u,v,w) := 
\sum_{p,q,r} \sum_{\alpha \in \Q/\Z} h^{p,q,r}(IH^{n - 1}_{\prim}(X_{\infty}))_\alpha \alpha u^p v^q w^r.
\]
Since primitive intersection cohomology is concentrated in $W$-degree equal to $\dim P-1$,
\begin{equation*}\label{e:pure}
E_{\inter, \prim}(X_\infty, \hat{\mu};u,v,w)  = w^{\dim P - 1} E_{\inter, \prim}(X_\infty, \hat{\mu};u,v,1).
\end{equation*}
One may then verify (see \cite[p.35]{Geometry}) that
\[
uvw^2E(X_\infty^\circ, \hat{\mu};u,v,w) = (uvw^2 - 1)^{\dim P} + (-1)^{\dim P + 1}(1  + uvw^2 E_{\prim}(X_\infty^\circ, \hat{\mu};u,v,w)), \]
\[
E_{\inter}(X_\infty, \hat{\mu};u,v,w)  = E_{\inter,\Lef}(P;uvw^2) + (-1)^{\dim P + 1}E_{\inter, \prim}(X_\infty, \hat{\mu};u,v,w),
\]
where $E_{\inter,\Lef}(P;t) $ is defined in \eqref{e:interlef}. We conclude that Theorem~\ref{t:mainintersection} is equivalent to the following:
\begin{equation}\label{e:rewrite}
uvE_{\inter, \prim}(X_\infty, \hat{\mu};u,v,1) = \lc(P,\nu;u,v).
\end{equation}
By Lemma~\ref{l:equiv}, it follows that 
Theorem~\ref{t:mainhyper} is equivalent to its specialization at $w = 1$:
\begin{equation}\label{e:mainhyper}
uvE(X_\infty^\circ, \hat{\mu};u,v) = (uv - 1)^{\dim P} + (-1)^{\dim P + 1} h^*(P,\nu;u,v).
\end{equation}

Before proceeding to the proof of Theorem~\ref{t:mainhyper} and Theorem~\ref{t:mainintersection}, we will need the following theorem. In fact, it is a consequence of the proof that the theorem below is a
special case of Theorem~\ref{t:mainhyper} and Theorem~\ref{t:mainintersection}.
An algorithm to compute the equivariant Hodge-Deligne polynomial below was given in
\cite[Section~2]{MTMonodromy}, extending an algorithm of Danilov and Khovanski{\u\i} in the non-equivariant setting \cite{DKAlgorithm}. 
The theorem below generalizes a formula of Borisov and Mavlyutov in the non-equivariant setting \cite{BMString}.

\begin{theorem}\label{t:DKformula}
Let $V^\circ =  \{ \sum_{u \in P \cap M} \lambda_u x^u = 0 \} \subseteq (\C^*)^n$ be a sch\"on complex hypersurface with Newton polytope $P$ and  $\dim P = n$.  Let $\nu$ be the convex graph of an integral height function on $P$,
and assume that $\nu$ is $\Q$-affine   
and $\nu(u) \in \Z$ whenever $\lambda_u \ne 0$.
Then
$V^\circ$ is invariant under multiplication by the associated element $\exp(-2\pi \sqrt{-1} \nu) \in (\C^*)^n$, and we have an induced good $\hat{\mu}$-action on $V^\circ$.
The corresponding equivariant Hodge-Deligne polynomial $E_{\hat{\mu}}(V^\circ;u,v) \in \Z[\Q/\Z][u,v]$ is given by:
\[
uvE_{\hat{\mu}}(V^\circ;u,v)  = (uv - 1)^{\dim P} + (-1)^{\dim P + 1} h^*(P,\nu;u,v).
\]
\end{theorem}
\begin{proof}
Recall from  Example~\ref{e:simpleaffine} that the variety $X^\circ = \{ f = \sum_{u \in P \cap M} \lambda_u t^{\nu(u)} x^u = 0 \} \subseteq (\K^*)^n$ is sch\"on, and satisfies
$\psi_{X^\circ} = [V^\circ \circlearrowleft \hat{\mu}]$, and hence $E(X^\circ_\infty, \hat{\mu}; u,v) = E_{\hat{\mu}}(V^\circ;u,v)$. Moreover, the action of monodromy on cohomology is semi-simple and hence
\[
E(X^\circ_\infty, \hat{\mu}; u,v,w) = E(X^\circ_\infty, \hat{\mu}; uw,vw) = E_{\hat{\mu}}(V^\circ;uw,vw).
\]
If 
$X = \cup_{Q \subseteq P} X^\circ_Q$ denotes the stratification in the statement of Theorem~\ref{t:sumoverstrata}, then $X^\circ_Q$ is defined by an equation of the form
$\{ \sum_{u \in Q \cap M} \lambda_u t^{\nu(u)} x^u = 0 \} $. It follows that if we restrict Theorem~\ref{t:mainhyper} to all $X^\circ$ that appear from complex sch\"on hypersurfaces $V^\circ$ as above, then
 Theorem~\ref{t:mainhyper}, Theorem~\ref{t:mainintersection} together with \eqref{e:rewrite} and \eqref{e:mainhyper} are all still equivalent when restricted to this class of varieties over $\K$.

With the notation above, \cite[Theorem~2.7]{MTMonodromy} states that
\[
 uE_{\hat{\mu}}(V^\circ;u,1) = (u - 1)^{\dim P} + (-1)^{\dim P + 1} h^*(P,\nu;u).
\]
This establishes \eqref{e:rewrite} and \eqref{e:mainhyper} and hence Theorem~\ref{t:mainhyper} and Theorem~\ref{t:mainintersection} for $X^\circ$ under the specialization $v = 1$.
We conclude that
\[
uw E_{\hat{\mu}}(V^\circ;u,w) =  E(X^\circ_\infty, \hat{\mu}; uw^{-1},1,w)  = (uw - 1)^{\dim P} + (-1)^{\dim P + 1} h^*(P,\nu;uw^{-1},1,w).
\]
By \eqref{prop7} in Theorem~\ref{t:combinatorics},  $h^*(P,\nu;uw^{-1},1,w) = h^*(P,\nu;u,w)$. This completes the proof.

\end{proof}


We now complete the proofs of Theorem~\ref{t:mainhyper} and Theorem~\ref{t:mainintersection}, which we have established are both equivalent to 
 the equivalent statements \eqref{e:rewrite} and \eqref{e:mainhyper}.
\begin{proof}
By Corollary~\ref{c:hyper},
\[
E(X^\circ_\infty, \hat{\mu}; u,v) =  \sum_{\substack{F \in \cS(\nu) \\\ \sigma(F) = P}}   E_{\hat{\mu}}( V_{F}^\circ; u,v)(1 - uv)^{\dim P - \dim F}.
\]
Substituting in Theorem~\ref{t:DKformula} yields:
\begin{align*}
uvE(X^\circ_\infty, \hat{\mu}; u,v) &=  \sum_{\substack{F \in \cS(\nu) \\\ \sigma(F) = P}} \big[ (uv - 1)^{\dim F} + (-1)^{\dim F + 1} h^*(F,\nu|_F;u,v) \big] (1 - uv)^{\dim P - \dim F} \\
&= (uv - 1)^{\dim P} +  (-1)^{\dim P + 1}  \sum_{\substack{F \in \cS(\nu) \\\ \sigma(F) = P}} h^*(F,\nu|_F;u,v) (uv - 1)^{\dim P - \dim F} \\
&= (uv - 1)^{\dim P} +  (-1)^{\dim P + 1}  h^*(P,\nu;u,v).
\end{align*}
Here the last equality follows from \eqref{prop6} in Theorem~\ref{t:combinatorics}. This establishes \eqref{e:mainhyper} as desired.
\end{proof}








Using Remark~\ref{r:existence} and the above proof, we immediately deduce the following geometric description of the combinatorial invariants introduced in Section~\ref{s:weighted}.
As in Remark~\ref{r:forget}, the corresponding result for invariants with coefficients in $\Z$ can be found in \cite[Corollary~5.11,Corollary~6.3]{Geometry}.

 \begin{corollary}\label{c:explicit}
 Let $P$ be a full-dimensional lattice polytope in a lattice $M$ and let $\nu: P \rightarrow \R$ be the convex graph of an integral height function on $P$. Let $X^\circ =  \{ f = 0 \} \subseteq \Spec \K[M]$
 be a sch\"on hypersurface with Newton polytope $P$ such that $\nu_f = \nu$. Let $X$ denote the closure of $X^\circ$ in the projective toric variety over $\K$ corresponding to the normal fan of $P$.
Then
\[
h^*(P,\nu;u,v,w) = 1 + uvw^2 \sum_{p,q,r} \sum_{\alpha \in \Q/\Z} h^{p,q,r}(H_{\prim,c}^{\dim P - 1}(X^\circ_\infty))_\alpha \alpha  u^p v^q w^r .
\]
In particular, this specializes to:
\[
h^*(P,\nu;u,v) = 1 + uv \sum_{p,q} \sum_{\alpha \in \Q/\Z} h^{p,q}(H_{\prim,c}^{\dim P - 1}(X^\circ_\infty))_\alpha \alpha u^p v^q ,
\]
\[
h^*(P,\nu;u) = 1 + u \sum_{p} \sum_{\alpha \in \Q/\Z} \dim \Gr_F^p (H_{\prim,c}^{\dim P - 1}(X^\circ_\infty))_\alpha \alpha u^p,
\]
\[
h^*(P,\nu;1) = 1 +  \sum_{\alpha \in \Q/\Z} \dim  H_{\prim,c}^{\dim P - 1}(X^\circ_\infty)_\alpha [\alpha].
\]
Moreover,
\[
\lc(P, \nu;u,v) = uv \sum_{p,q} \sum_{\alpha \in \Q/\Z}  h^{p,q}(IH_{\prim}^{\dim P - 1}(X_\infty))_\alpha \alpha u^p v^q.
\]
This specializes to:
\[
\lc(P, \nu;u) = u \sum_{p} \sum_{\alpha \in \Q/\Z}   \dim \Gr_F^p(IH_{\prim}^{\dim P - 1}(X_\infty))_\alpha \alpha u^p,
\]
\[
\lc(P, \nu;1) =  \sum_{\alpha \in \Q/\Z}   \dim IH_{\prim}^{\dim P - 1}(X_\infty)_\alpha [\alpha].
\]
 \end{corollary}

\begin{example}\label{e:orbifold}
Let $P$ be a lattice polytope in a lattice $N$ such that $P$ contains the origin in its relative interior, and let $\{ b_i \}$ denote the vertices of $P$. Let $\nu: P \rightarrow \R$ be the piecewise $\Q$-affine function with value $0$ at the origin and value $1$ on the boundary of $P$.  
For an explicit example, see Example~\ref{e:old}.
Let $X^\circ =  \{ f = 0 \} \subseteq \Spec \K[N]$
 be a sch\"on hypersurface with Newton polytope $P$ such that $\nu_f = \nu$. For example, if $g$ is a sch\"on complex polynomial with Newton polytope $P$, then take $f = tg - \lambda$ for a generic choice of $\lambda \in \C^*$. Define the \define{primitive spectrum} of $f$ to be:
\[
\Sp_{f}^{\prim}(t) := 1 + \sum_{\beta \in  (0,1] \cap \Q} \sum_{p,q} h^{p,q}(H^{n - 1}_{\prim,c}(X^\circ_{\infty}))_{\exp(2\pi \sqrt{-1} \beta)} t^{p + \beta}.
\]

Consider the weighted $h^*$-polynomial $h^*(P,\nu; u) = 1 + u \sum_{\beta \in (0,1]}  \sum_p h^*_{p,\beta} [\beta] u^p \in \Z[\Q/\Z][u]$, for some non-negative integers $h^*_{p,\beta}$.
We consider an alternative way of encoding this polynomial.
Specifically, for some sufficiently divisible positive integer $N$, define
\[
\tilde{h}(P; t) := 1 + \sum_{\beta \in (0,1]}  \sum_p h^*_{p,\beta} t^{p + \beta} \in \Z[t^{1/N}].
\]
This was the definition of the weighted $h^*$-polynomial given in \cite{Weighted}. 
Note that the $h^*$-polynomial $h^*(P;t)$ is obtained from $\tilde{h}(P; t)$ by rounding up the exponents of $t$ to the nearest integer.
By \cite[Corollary~2.12]{Weighted}, we have the symmetry:
\[
\tilde{h}(P; t)  = t^{\dim P} \tilde{h}(P; t^{-1}).
\]

Consider the complete fan $\Sigma'$ in $N_\R$ with cones given by the cones over the proper faces of $P$, and let $\Sigma$ be a simplicial fan refinement of $\Sigma'$ with the same rays. Then the
  the vertices $\{ b_i \}$  of $P$ constitute a choice of a non-zero lattice point on each ray of $\Sigma$. As in Remark~\ref{r:old}, the triple $\bSigma = (N,\Sigma,\{ b_i \})$ is a stacky fan and we may consider the
  corresponding Deligne-Mumford stack $\mathcal{X} = \mathcal{X}(\bSigma)$ with coarse moduli space
the toric variety associated to the fan $\Sigma$ \cite{BCSOrbifold}.
The theory of orbifold cohomology, developed by Chen and Ruan \cite{CROrbifold, CRNew},
associates to $\mathcal{X}$ a finite-dimensional $\mathbb{Q}$-algebra
$H_{\orb}^{*}(\mathcal{X}, \mathbb{Q})$, graded by $\mathbb{Q}$.





By \cite[Theorem~4.3]{Weighted}, together with Corollary~\ref{c:explicit}, we see that the primitive spectrum of $f$ coincides with the Betti polynomial of $\mathcal{X}$:
\[
\Sp_{f}^{\prim}(t) = \sum_{i \in \Q}  \dim H_{\orb}^{2i}(\mathcal{X}, \mathbb{Q}) t^i = \tilde{h}(P; t).
\]
For an explicit example (cf.  \cite[Example~1.1]{Weighted}), in
Example~\ref{e:old} we computed:
\[ h^*(P,\nu; u)  = 1+ 4u + u^{2}    + 2u(1 + u)[1/2] + u[2/3]  + u^2[1/3].      \]
Equivalently,
\[
\tilde{h}(P; t) = 1 + 2t^{1/2} + t^{2/3} + 4t + t^{4/3} + 2t^{3/2}+ t^2.
\]

\end{example}

\section{Applications to 
the monodromy of complex polynomials}\label{s:sing}

In this section, we apply Theorem~\ref{t:mainhyper} to deduce explicit combinatorial equations for some important invariants associated to  the monodromy of complex polynomials. 
Recall that we identify the group  $\Q/\Z$ with the group $\mathbb{S}^1_\Q$ of rational points on the circle $\{ z \in \C \mid |z| = 1 \}$, sending $[k] \in \Q/\Z$ to $\alpha = e^{2 \pi \sqrt{-1}k} \in \mathbb{S}^1_\Q$.
We will continue with the notation of Section~\ref{s:degenerations}, Section~\ref{s:weighted} and Section~\ref{s:hyper}, and set $M = \Z^n$.

\subsection{Degenerations of hypersurfaces in affine space} \label{s:general}

We first consider applications of the results of Section~\ref{s:hyper} for families of hypersurfaces of affine space.
Let $X^\circ \subseteq (\K^*)^n$ be a sch\"{o}n hypersurface, with associated Newton polytope and convex graph of an integral height function $(P,\nu)$  and $\dim P = n$.
Assume that $P \subseteq \R_{\ge 0}^n$,
 and for each (possibly empty) subset $S$ of $\{1,\ldots,n\}$, let
$\R^S = \{ (v_1,\dots,v_n) \mid v_i = 0 \textrm{ if } i \notin S \}$, $\Z^S = \R^S \cap  \Z^n$ and $P^S = P \cap \R^S$. We assume that $P$ is \define{convenient} in the sense that $\dim P^S$ equals the cardinality   $| S |$ of $S$ for every
subset  $S$ of $\{1,\ldots,n\}$. Equivalently, $P$ contains the origin and has non-zero intersection with each  ray through a coordinate vector.
We call a face (including the empty face) of $P$ that does not contain the origin, a \define{face at infinity}, and write $P_\infty$ for the union of faces at infinity of $P$.
Then the faces of $P$ are precisely the faces at infinity together
with $\{ P^S \mid S \subseteq \{ 1, \ldots, n\} \}$. Let $X$ denote the closure of $X^\circ$ in $\K^n$. Then $X_{\gen} \subseteq \C^n$ is a smooth hypersurface and, as in Section~\ref{s:hyper}, one
has
a weak Lefschetz result  (see, for example, \cite[Corollary~3.8]{DKAlgorithm}) stating that $H^{2(n - 1)}_c(X_{\gen}) = \C$ with trivial monodromy action and the only other non-zero cohomology is in middle-dimension
$H^{n - 1}_c(X_{\gen})$.
By Corollary~\ref{c:hyper},  the  motivic nearby fiber of $X$ is given by
\begin{equation}\label{e:mnf}
\psi_{X} =  \sum_{S \subseteq \{ 1, \ldots, n \} } \sum_{\substack{F \in \cS(\nu) \\\ \sigma(F) = P^S}}  [V_{F}^\circ \circlearrowleft \hat{\mu}](1 - \L)^{|S| - \dim F},
\end{equation}
where $\sigma(F)$ is the smallest face of $P$ containing $F$.
After rearranging terms, Theorem~\ref{t:mainhyper} implies that
\begin{equation}\label{e:refined}
uvw^2E(X_\infty, \hat{\mu};u,v,w) =  (uvw^2)^n + (-1)^{n - 1} \sum_{S \subseteq \{ 1, \ldots, n \} } (-1)^{n - |S|} h^*(P^S, \nu|_{P^S}; u,v,w).
\end{equation}
Specializing by setting $w = 1$ gives a formula for the equivariant limit Hodge-Deligne polynomial. We may specialize further by setting $v = 1$ and then $u = 1$. In particular,
by Example~\ref{e:volume},
we have  the following formula for the eigenvalues (with multiplicity) of the action of monodromy on the cohomology of $X_{\gen}$:
\begin{equation}\label{e:char}
\sum_{\alpha \in \Q/\Z} \dim  H_{c}^{n - 1}(X_\infty)_\alpha \alpha =  \sum_{S \subseteq \{ 1, \ldots, n \} }  \sum_{ \substack{F \in \cS(\nu)|_{P^S} \\ \dim F = \dim P^S} }   (-1)^{n - |S|} \frac{\Vol(F)}{m_F} \sum_{i = 0}^{m_F - 1} [i/m_F],
\end{equation}
where $m_F$ is the minimal positive integer such that  $m_F \nu|_{F}$ is an affine 
function with respect to $\Z^S$. Here $\Vol(F)$ is the normalized volume of $F$ and $\Vol(F)/m_F$ is a positive integer.

\begin{remark}\label{r:notconvenient}
With the setup above, do not assume that $P$ is convenient, and
instead
only assume that $P$ contains the origin.
Then we may not apply the same weak Lefschetz argument.
However, we still obtain the
following formulas for the motivic nearby fiber and
equivariant refined limit  Hodge-Deligne polynomial:
\[
\psi_{X} =
\sum_{  S \subseteq \{ 1, \ldots, n \} } (-1)^{|S| - \dim P^S} \sum_{\substack{F \in \cS(\nu) \\\ \sigma(F) = P^S}}  [V_{F}^\circ \circlearrowleft \hat{\mu}](1 - \L)^{|S| - \dim F}.
\]
\[
uvw^2E(X_\infty, \hat{\mu};u,v,w) = (uvw^2)^n + \sum_{ S \subseteq \{ 1, \ldots, n \}  }  (-1)^{\dim P^S - 1} (uvw^2 - 1)^{|S| - \dim P^S} h^*(P^S, \nu|_{P^S}; u,v,w).
\]
Specializing at $u = v = w =  1$ gives:
\begin{align*}
E(X_\infty, \hat{\mu};1,1,1) &= \sum_m \sum_{\alpha \in \Q/\Z} (-1)^m \dim  H_{c}^{m}(X_\infty)_\alpha \alpha \\
&= \sum_{ \substack{ \emptyset \ne S \subseteq \{ 1, \ldots, n \} \\ \dim P^S = |S|} } (-1)^{|S| - 1}
 \sum_{ \substack{F \in \cS(\nu)|_{P^S} \\ \dim F = \dim P^S} }   \frac{\Vol(F)}{m_F} \sum_{i = 0}^{m_F - 1} [i/m_F],
\end{align*}
where $m_F$ is the minimal positive integer such that  $m_F \nu|_{F}$ is an affine 
function with respect to $\Z^S$. 
\end{remark}

\emph{Assume further that the restriction of $\nu$ to $P_\infty$ 
is constant}.
Substitute in Definition~\ref{d:new} and observe that $g([P^S, P^{S'}]; t) = 1$ for any $S \subseteq S'$ by Example~\ref{e:boolean}.
By assumption, for every face $Q$ at infinity, $\lc(Q, \nu|_Q;u,v) = v^{\dim Q + 1} \lc(Q;uv^{-1})$.
 In this case,  we may rewrite \eqref{e:refined} as:
\[
uvw^2E(X_\infty, \hat{\mu};u,v,w)  = (uvw^2)^n + (-1)^{n - 1} \big[ F(uw,vw) +w^{n + 1} \lc(P,\nu;u,v) \big],
\]
where
\[
F(u,v)  :=  \sum_{ Q \subseteq P_\infty  } v^{\dim Q + 1} \lc(Q;uv^{-1}) G(Q;uv),
\]
\[
G(Q;t) := \sum_{ \substack{ S \subseteq \{1,\ldots,n\} \\  Q \subseteq   P^S} } (-1)^{n - |S|}g([Q,P^S]; t).
\]
Using Theorem~\ref{t:Eulerianinverse} and Example~\ref{e:trivialdecomp}, we have:
\begin{align*}
G(Q;t) &= \sum_{Q \subseteq Q' \subseteq P_\infty} (-1)^{n - 1 - \dim Q'} g([Q,Q']; t)g([Q',P]^*; t) \\
&= \big[ \sum_{ \substack{ S \subseteq \{1,\ldots,n\} \\  Q \subseteq   P^S} } (-1)^{n - |S|}  h(\lk_{\cS(\nu)|_{P^S}}(Q);t) \big] - l_P(\cS(\nu),Q;uv). 
\end{align*}

\begin{example}\label{e:G}
With the notation above, let $Q$ be a face at infinity and consider $G(Q;t)$.  By definition, $G(Q;t)$ has degree strictly less than $(\dim P - \dim Q)/2$.
If $\Int (Q) \subseteq \R_{> 0}^n$, then $G(Q;t) = g([Q,P];t)$ has constant term $1$. Otherwise, $G(Q;t)$ has no constant term.
Using Example~\ref{e:linear}, when $Q = \emptyset$, we compute that  the linear coefficient of $G(Q;t)$ is the number of vertices of $P$ contained in $\R_{> 0}^n$.

One then computes that $F(u,v)$ has the form
\[
F(u,v) = \big[  \sum_{ \substack{ Q \subseteq P_\infty \\ \dim Q \le 1 \\ \Int(Q) \subseteq \R_{> 0}^n}  }   \#( \Int(Q) \cap \Z^n)\big]  uv + \big[  \sum_{ \substack{ Q \subseteq P_\infty \\ \dim Q =2 \\ \Int(Q) \subseteq \R_{> 0}^n}  }
 \#( \Int(Q) \cap \Z^n)\big]  uv(u + v) + \beta(u,v),
\]
where every term in  $\beta(u,v)$ has combined degree in $u$ and $v$ at least $4$.
\end{example}

Observe that every term in $F(uw,vw)$ has combined degree in $u$ and $v$ equal to its degree in $w$. Hence, the only contribution above corresponding to a non-trivial action of monodromy
 is
$(-w)^{n+ 1} \lc(P,\nu;u,v)$. We conclude that we have the following corollary.


\begin{corollary}\label{c:affine}
Let $X^\circ \subseteq (\K^*)^n$ be a sch\"{o}n hypersurface, with associated Newton polytope and convex graph of an integral height function $(P,\nu)$  and $\dim P = n$.
Let $X \subseteq \K^n$ denote the closure of $X^\circ$ in $\K^n$.
Assume that $P \subseteq \R_{\ge 0}^n$ is convenient, and  the restriction of $\nu$ to every face at infinity is constant. Then the action of monodromy is trivial on the
graded pieces $\Gr^W_{r} H^{n - 1}_c (X_{\gen})$ of the Deligne weight filtration for $r \ne n - 1$, and we have the  following explicit combinatorial formula for the equivariant limit mixed Hodge numbers
of $\Gr^W_{ n - 1} H^{n - 1}_c (X_{\gen})$:
\[
uv \sum_{p,q} \sum_{\alpha \in \Q/\Z} h^{p,q}(\Gr^W_{n - 1} H^{n - 1}_c (X_{\infty}))_{\alpha} \alpha u^p v^q  =  \lc(P,\nu; u,v).
\]
\end{corollary}
In particular, one may use Definition~\ref{d:weight} to give a formula for the Jordan block structure of the action of monodromy on $\Gr^W_{ n - 1} H^{n - 1}_c (X_{\gen})$.

\begin{example}\label{e:dim2}
With the notation of Corollary~\ref{c:affine}, let $n = 2$ and consider $X \subseteq \K^2$. Using Example~\ref{e:smallterms} and Example~\ref{e:G}, we compute
\[
E(X_\infty, \hat{\mu};u,v,w)  = uvw^2 - \big[ b +w[ h^*_{0,0,1}(P,\nu)(1 + uv) + h^*_{0,1,1}(P,\nu)v + \bar{h^*_{0,1,1}(P,\nu)}u]   \big],
\]
where $b = \#(\partial P \cap \Z^{2}_{> 0})$, $\partial P$ denotes the boundary of $P$, and 
\[
h^*_{0,0,1}(P,\nu) = \sum_{\substack{F \in \cS(\nu) \\\ \dim F \le 1 \\\ \sigma(F) = P}} \sum_{v \in \Int(F) \cap \Z^2} w(v), \; \;
h^*_{0,1,1}(P,\nu) =  \sum_{\substack{F \in \cS(\nu) \\\ \dim F = 2 \\\ \sigma(F) = P}} \sum_{v \in \Int(F) \cap \Z^2} w(v).
\]
In particular, the action of monodromy on $\Gr^W_{0} H^{1}_c (X_{\gen})$ is trivial and $\dim \Gr^W_{0} H^{1}_c (X_{\gen}) = b$.
The equivariant limit mixed Hodge numbers of  $\Gr^W_{1} H^{1}_c (X_{\gen})$ are given by
\[
\sum_{\alpha \in \Q/\Z} h^{0,0}(\Gr^W_{1} H^{1}_c (X_{\infty}))_\alpha \alpha = \sum_{\alpha \in \Q/\Z} h^{1,1}(\Gr^W_{1} H^{1}_c (X_{\infty}))_\alpha \alpha = h^*_{0,0,1}(P,\nu),
\]
\[
\sum_{\alpha \in \Q/\Z} h^{0,1}(\Gr^W_{1} H^{1}_c (X_{\infty}))_\alpha \alpha = \sum_{\alpha \in \Q/\Z} h^{1,0}(\Gr^W_{1} H^{1}_c (X_{\infty}))_\alpha \alpha^{-1} = h^*_{0,1,1}(P,\nu).
\]
The Jordan normal form of the action of monodromy on $\Gr^W_{1} H^{1}_c (X_{\gen})$ has
\[ h^{0,1}(\Gr^W_{1} H^{1}_c (X_{\infty}))_\alpha + h^{1,0}(\Gr^W_{1} H^{1}_c (X_{\infty}))_\alpha\]
Jordan blocks of size $1$ with eigenvalue $\alpha$, and
$h^{0,0}(\Gr^W_{1} H^{1}_c (X_{\infty}))_\alpha $
Jordan blocks of size $2$ with eigenvalue $\alpha$.
\end{example}

\begin{remark}\label{r:limit}
Specializing by setting $w = 1$ above gives a formula for the equivariant limit mixed Hodge numbers:
\[
uv \sum_{p,q} \sum_{\alpha \in \Q/\Z} h^{p,q}( H^{n - 1}_c (X_{\infty}))_{\alpha} \alpha u^p v^q  = F(u,v) +  \lc(P,\nu; u,v).
\]
Substituting in the definitions and simplifying, 
the right hand side may be rewritten as follows:
\[
\sum_{\substack{ F \in \cS(\nu) \\  F \nsubseteq P_\infty 
} } v^{\dim F + 1}\lc(F, \nu|_F; uv^{-1})  l_P(\cS(\nu),F;uv) 
\]
\[
 + 
 \sum_{Q \subseteq P_\infty}
v^{\dim Q + 1} \lc(Q;uv^{-1}) \sum_{ \substack{ S \subseteq \{1,\ldots,n\} \\  Q \subseteq   P^S} } (-1)^{n - |S|} h(\lk_{\cS(\nu)|_{P^S}}(Q);uv).
\]
\end{remark}



We present the following two 
 examples of the above situation. Below we let $f \in \C[x_1, \ldots, x_n]$ be a complex polynomial with Newton polytope $\NP(f)$.  For every face $Q$ of $\NP(f)$, we let
$\Delta_Q$ denote the convex hull of $Q$ and the origin. We set $P = \Delta_{\NP(f)}$.
We assume that $f$ is \define{convenient} in the sense that $\NP(f)$  has non-zero intersection with each  ray through a coordinate vector \cite{Kou}.  Recall that for a complex hypersurface of a torus,
sch\"on is also called \emph{non-degenerate}. Let $\Gamma_+(f) = \NP(f) + \R_{\ge 0}^n$ denote the \define{Newton polyhedron} of $f$.
Observe that every bounded face of $\Gamma_+(f)$ is a  face of $NP(f)$. We let $\Gamma_f$ denote the union of the bounded faces of $\Gamma_+(f)$.

\begin{remark}\label{non-zero}
For any proper face $Q$ of $\NP(f)$ not containing the origin, consider a $\Q$-affine function $\nu$ on $\Delta_Q$ with value $m \pm 1$ at the origin, and value $m$ on $Q$, for some integer $m$. Consider $\lc(Q,\nu;u) \in \Z[\Q/\Z][u]$
as a sum of integer valued polynomials indexed by $\alpha \in \Q/\Z$. Then the coefficient of $\alpha = 1$ is zero. To see this, one may for example use Remark~\ref{r:decompose} to reduce to the case when
$Q$ is a simplex,
and then apply Example~\ref{e:simplex}.
\end{remark}

\begin{example}[Monodromy at $0$]\label{e:0}
Assume that $f \in \C[x_1, \ldots, x_n]$ has no constant term, $f$ is convenient and  $\{ f =  0 \} \subseteq (\C^*)^n$ defines a sch\"on hypersurface. Then for a general choice of $\lambda \in \C^*$,  $X^\circ = \{ f = \lambda t \} \subseteq (\K^*)^n$ defines a sch\"{o}n,
convenient hypersurface with Newton polytope $P$, and $\nu_0 := \nu$ is the piecewise $\Q$-affine function on $P$ with value $1$ at the origin and value identically $0$ on $NP(f)$.
The cells of $\cS(\nu_0)$ are given by the union of  $\{ Q \mid Q \subseteq \NP(f) \}$  and $\{ \Delta_Q \mid Q \subseteq \Gamma_f \}$. By \eqref{e:mnf},
we have the following equation for the motivic nearby fiber:
\begin{equation}\label{e:mnf0}
\psi_{X} =  \sum_{Q \subseteq \Gamma_f }  [V_{\Delta_Q}^\circ \circlearrowleft \hat{\mu}](1 - \L)^{\dim \sigma(\Delta_Q) - \dim \Delta_Q} +  \sum_{\substack{Q \subseteq \NP(f) \\\ Q \nsubseteq P_\infty}}  [V_{Q}^\circ](1 - \L)^{\dim \sigma(Q) - \dim Q}.
\end{equation}
By Corollary~\ref{c:affine}, the  action of monodromy is trivial on the
graded pieces $\Gr^W_{r} H^{n - 1}_c (X_{\gen})$ of the Deligne weight filtration for $r \ne n - 1$, and the equivariant limit mixed Hodge numbers
of $\Gr^W_{ n - 1} H^{n - 1}_c (X_{\gen})$ are given by:
\[
uv \sum_{p,q} \sum_{\alpha \in \Q/\Z} h^{p,q}(\Gr^W_{n - 1} H^{n - 1}_c (X_{\infty}))_{\alpha} \alpha u^p v^q  =  \lc(P,\nu_0; u,v).
\]
By Remark~\ref{r:limit} and Remark~\ref{non-zero}, we have the following formulas for the equivariant limit mixed Hodge numbers of $H^{n - 1}_c (X_{\gen})$:
\[
uv \sum_{p,q} \sum_{ \substack{ \alpha \in \Q/\Z \\ \alpha \ne 1} } h^{p,q}( H^{n - 1}_c (X_{\infty}))_{\alpha} \alpha u^p v^q  =  \sum_{Q \subseteq \Gamma_f } v^{\dim \Delta_Q + 1}\lc( \Delta_Q , \nu_0|_{ \Delta_Q }; uv^{-1})   l_P(\cS(\nu_0),\Delta_Q;uv), 
\]
\begin{align*}
uv \sum_{p,q} h^{p,q}( H^{n - 1}_c (X_{\infty}))_{1}  u^p v^q  &=  \sum_{\substack{Q \subseteq \NP(f) \\\ Q \nsubseteq P_\infty}}
v^{\dim Q + 1}\lc( Q , \nu_0|_{ Q }; uv^{-1})  
l_P(\cS(\nu_0), Q;uv),  \\ &+  \sum_{Q \subseteq P_\infty}
v^{\dim Q + 1} \lc(Q;uv^{-1}) \sum_{ \substack{ S \subseteq \{1,\ldots,n\} \\  Q \subseteq   P^S} } (-1)^{n - |S|} h(\lk_{\cS(\nu_0)|_{P^S}}(Q);uv).
\end{align*}
By \eqref{e:char}, we have  a formula for the eigenvalues (with multiplicity) of the action of monodromy on the cohomology of $X_{\gen}$. Explicitly,
$\sum_{\alpha \in \Q/\Z} \dim  H_{c}^{n - 1}(X_\infty)_\alpha \alpha$ is equal to
\begin{equation}\label{e:0eigen}
 \sum_{  \substack{ Q \subseteq \Gamma_f \\ \dim \sigma(\Delta_Q) = \dim \Delta_Q}} (-1)^{n - 1 - \dim Q} \Vol(Q) \sum_{i = 0}^{m(Q) - 1} [i/m(Q)] +
\sum_{\substack{Q \subseteq \NP(f) \\\ Q \nsubseteq P_\infty \\\ \dim \sigma(Q) = \dim Q}} (-1)^{n - \dim Q} \Vol(Q),
\end{equation}
where $\Vol(Q)$ is the normalized volume of $Q$ and $m(Q)$ is the \define{lattice distance} of $Q$ from the origin i.e. $m(Q)$ is
the minimal positive integer such that  $m(Q) \nu_0|_{\Delta_Q}$ is an affine 
function with respect to the lattice given by intersecting $M$ with the affine span of $\Delta_Q$. When $Q$ is empty,
$\Vol(Q) = m(Q) = 1$ and $\dim Q = -1$.

\end{example}


\begin{example}[Monodromy at infinity]\label{e:infty}
Assume that $f \in \C[x_1, \ldots, x_n]$ is convenient. Assume that $f$ is \define{sch\"on at infinity},  meaning that  $\{ f|_Q =  0 \} \subseteq (\C^*)^n$ defines a smooth hypersurface whenever $Q$ is a face of $P$ at infinity.
Then for a general choice of $\lambda \in \C^*$,  $X^\circ = \{ ft = \lambda \} \subseteq (\K^*)^n$ defines a sch\"{o}n,
convenient hypersurface with Newton polytope $P$, and $\nu_\infty := \nu$ is the piecewise $\Q$-affine function on $P$ with value $0$ at the origin and value identically $1$ on the faces at infinity of $P$.
The cells of $\cS(\nu_\infty)$ are given by the union of  $\{ Q \mid Q \subseteq P_\infty \}$  and $\{ \Delta_Q \mid Q \subseteq P_\infty \}$.
 By \eqref{e:mnf},  we have the following equation for the motivic nearby fiber:
\[
\psi_{X} =  \sum_{Q \subseteq P_\infty}  [V_{\Delta_Q}^\circ \circlearrowleft \hat{\mu}](1 - \L)^{\dim \sigma(\Delta_Q) - \dim \Delta_Q}.
\]
The \emph{motivic nearby fiber at infinity} of a convenient polynomial $f$ was introduced independently and from different perspectives  in \cite{MTMonodromy} and \cite{Rai}, and an explicit formula was given in \cite[Theorem~5.3]{MTMonodromy},
which agrees with our formula for the motivic nearby fiber above. In particular, the motivic nearby fiber at infinity is equal to the usual motivic nearby fiber of $X$.

Let $\Delta^\infty$ be the simplex $P \cap \{ x_1 + \cdots + x_n = \epsilon \}$ for
fixed $\epsilon > 0$ sufficiently small. Let $\cS_\infty$ 
 denote the regular, rational polyhedral subdivision obtained by intersecting $\cS(\nu_\infty)$ with $\Delta^\infty$. That is, the cells of $\cS_\infty$ are
$\{ Q_\infty := \Delta_Q \cap \Delta^\infty \mid Q \subseteq P_\infty \}$.
By Remark~\ref{r:limit} and Remark~\ref{non-zero}, we have the following formulas for the equivariant limit mixed Hodge numbers of $H^{n - 1}_c (X_{\gen})$:
\[
uv \sum_{p,q} \sum_{ \substack{ \alpha \in \Q/\Z \\ \alpha \ne 1} } h^{p,q}( H^{n - 1}_c (X_{\infty}))_{\alpha} \alpha u^p v^q  = \sum_{Q \subseteq P_\infty}  v^{\dim \Delta_Q + 1}\lc(\Delta_Q , \nu_\infty|_{\Delta_Q }; uv^{-1})  
l_{\Delta^\infty}(\cS_\infty,Q_\infty;uv),
\]
\[
uv \sum_{p,q} h^{p,q}( H^{n - 1}_c (X_{\infty}))_{1}  u^p v^q  =  \sum_{Q \subseteq P_\infty}
v^{\dim Q + 1} \lc(Q;uv^{-1}) l_{\Delta^\infty}(\cS_\infty,Q_\infty;uv). 
\]
 An algorithm to compute the equivariant limit mixed Hodge numbers above was given in
\cite[Section~5]{MTMonodromy}.
\begin{remark}\label{r:termsym}
Recall from Remark~\ref{r:non-negative} that $ l_{\Delta^\infty}(\cS_\infty,Q_\infty;t)  = t^{n - 1 - \dim Q}  l_{\Delta^\infty}(\cS_\infty,Q_\infty;t^{-1})$ has non-negative, symmetric, unimodal coefficients.
Also, by Remark~\ref{r:symmetry}, the coefficients of the local weighted $h^*$-polynomial are non-negative, and we have symmetries $\lc(Q;u) = u^{\dim Q + 1} \lc(Q;u^{-1})$
and $\lc(\Delta_Q ,\nu_\infty|_{\Delta_Q };u) = u^{\dim \Delta_Q  + 1} \bar{\lc(\Delta_Q ,\nu_\infty|_{\Delta_Q };u^{-1})}$.
\end{remark}

Recall also from \eqref{e:refined} that we have the following equivalent formula
for the equivariant refined limit Hodge-Deligne polynomial:
\[
uvw^2E(X_\infty, \hat{\mu};u,v,w) =  (uvw^2)^n + (-1)^{n - 1} \sum_{S \subseteq \{ 1, \ldots, n \} } (-1)^{n - |S|} h^*(P^S, \nu_\infty|_{P^S}; u,v,w).
\]
After specializing at $v = w = 1$, one obtains a formula for the \emph{spectrum at infinity} of $f$ \cite{SabMonodromy} that is equivalent to \cite[Theorem~5.11]{MTMonodromy}.
Recall from \eqref{e:char} that specialization at $u = v = w = 1$ yields
a formula for the eigenvalues (with multiplicity) of the action of monodromy:
\[
\sum_{\alpha \in \Q/\Z} \dim  H_{c}^{n - 1}(X_\infty)_\alpha \alpha = \sum_{  \substack{ Q \subseteq P_\infty \\ \dim \sigma(\Delta_Q) = \dim \Delta_Q}} (-1)^{n - 1 - \dim Q} \Vol(Q) \sum_{i = 0}^{m(Q) - 1} [i/m(Q)],
\]
where $\Vol(Q)$ is the normalized volume of $Q$ and $m(Q)$ is the lattice distance of $Q$ from the origin i.e. $m(Q)$ is
the minimal positive integer such that  $m(Q) \nu_\infty|_{\Delta_Q}$ is an affine 
function with respect to the lattice given by intersecting $M$ with the affine span of $\Delta_Q$. When $Q$ is empty,
$\Vol(Q) = m(Q) = 1$ and $\dim Q = -1$.
The above is equivalent  to a formula of  Libgober  and Sperber
\cite{LSZeta} for the \emph{zeta function at infinity} of $f$.

\end{example}


\begin{example}\label{e:simpleagain}
Following on with Example~\ref{e:simple}, Example~\ref{e:simple2} and Example~\ref{e:simple3}, let $f  = \sum_{i  = 1}^n 
x_i^{m_i} \in  \C[x_1,\ldots,x_n]$, for some  
$m_i \in \Z_{>0}$.
Let $f_1, \ldots, f_n$ denote
the standard basis vectors of $M_\R =  \R^n$. Then $P$ is the convex hull of the origin and $\{ m_i f_i \mid 1 \le i \le n \}$.
In this case, $f: \C^n \rightarrow \C$ is a locally trivial fibration away from the origin, and it follows that the actions of monodromy at $0$ and monodromy at infinity
on cohomology are inverse to each other. In both cases, $\nu$ is $\Q$-affine and $\cS(\nu)$ is trivial.
For $(i_1, \ldots i_n) \in [1,m_1 - 1] \times \cdots \times [1,m_n - 1]$, let
\[
\nu_\infty(i_1,\ldots,i_n) := i_1/m_1 + \cdots + i_n/m_n.
\]
The action of monodromy is semi-simple, and, using Example~\ref{e:simplex2}, we compute the equivariant refined limit Hodge-Deligne polynomial $uvw^2E(X_\infty, \hat{\mu};u,v,w)$ for monodromy at infinity:
\[
(uvw^2)^n + (-1)^{n - 1} \sum_{(i_1, \ldots i_n) \in [1,m_1 - 1] \times \cdots \times [1,m_n - 1]} [ \nu_\infty(i_1,\ldots,i_n)] (uw)^{\lceil  \nu_\infty(i_1,\ldots,i_n) \rceil}
(vw)^{\lceil  n - \nu_\infty(i_1,\ldots,i_n) \rceil}.
\]
For the corresponding invariant for monodromy at $0$, one simply reverses the roles of $u$ and $v$.


\end{example}

\begin{example}\label{e:double}
Let $f = a_0 x^4 + a_1 x^5 + a_2 x^4y^2 + a_3 xy^5 + a_4 y^5 + a_5 xy^2 + a_6 x^2y \in \C[x,y]$ for a general choice of  $a_0, \ldots, a_6 \in \C^*$.
Then the pairs $(P, \nu)$ with their corresponding subdivisions $\cS(\nu)$ in Example~\ref{e:0} and Example~\ref{e:infty} are shown below. We have labelled all interior lattice points with non-zero weight.
We may calculate the invariants above using Example~\ref{e:dim2}. In particular, using the notation of Example~\ref{e:dim2}, we have
  $\dim \Gr^W_{0} H^{1}_c (X_{\gen})  = b =  4$. Also, for monodromy at $0$, we have
\[
h^*_{0,0,1}(P,\nu_0) = 2, \; h^*_{0,1,1}(P,\nu_0) = 7 + [1/3].
\]
In this case, the action of monodromy on  $\Gr^W_{1} H^{1}_c (X_{\gen})$ has 16 Jordan blocks of size $1$, $14$ with eigenvalue $1$ and $1$ with
 eigenvalue $\exp(2\pi \sqrt{-1}/3)$ and  $\exp(4\pi \sqrt{-1}/3)$ respectively, as well as $2$ Jordan blocks of size $2$ with eigenvalue $1$.
 For monodromy at infinity, we have
 \[
 h^*_{0,0,1}(P,\nu_\infty) = [1/2], \; h^*_{0,1,1}(P,\nu_\infty) = [7/10] + [9/10] + [1/3] + [1/2]  + 2[2/3] + 3[5/6].
 \]
 In this case, the action of monodromy on  $\Gr^W_{1} H^{1}_c (X_{\gen})$ has 18 Jordan blocks of size $1$, $3$ with eigenvalues $\exp(2\pi \sqrt{-1}/6)$,
 $\exp(2\pi \sqrt{-1}/3)$, $\exp(4\pi \sqrt{-1}/3)$ and  $\exp(10\pi \sqrt{-1}/6)$ respectively, $2$ with eigenvalue $-1$ and $1$ with eigenvalues
  $\exp(\pi \sqrt{-1}/5)$,
 $\exp(3\pi \sqrt{-1}/5)$, $\exp(7\pi \sqrt{-1}/5)$ and  $\exp(9\pi \sqrt{-1}/5)$ respectively, as well as a single Jordan block of size $2$ with eigenvalue $-1$.


 \vspace{1cm}

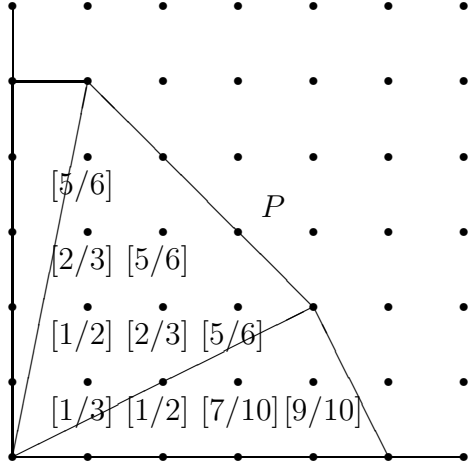
\setlength{\unitlength}{1cm}
\begin{picture}(12,6.5)


\put(0,0){\circle*{0.1}}
\put(0,1){\circle*{0.1}}
\put(0,2){\circle*{0.1}}
\put(0,3){\circle*{0.1}}
\put(0,4){\circle*{0.1}}
\put(0,5){\circle*{0.1}}
\put(0,6){\circle*{0.1}}
\put(1,0){\circle*{0.1}}
\put(1,1){\circle*{0.1}}
\put(1,2){\circle*{0.1}}
\put(1,3){\circle*{0.1}}
\put(1,4){\circle*{0.1}}
\put(1,5){\circle*{0.1}}
\put(1,6){\circle*{0.1}}
\put(2,0){\circle*{0.1}}
\put(2,1){\circle*{0.1}}
\put(2,2){\circle*{0.1}}
\put(2,3){\circle*{0.1}}
\put(2,4){\circle*{0.1}}
\put(2,5){\circle*{0.1}}
\put(2,6){\circle*{0.1}}
\put(3,0){\circle*{0.1}}
\put(3,1){\circle*{0.1}}
\put(3,2){\circle*{0.1}}
\put(3,3){\circle*{0.1}}
\put(3,4){\circle*{0.1}}
\put(3,5){\circle*{0.1}}
\put(3,6){\circle*{0.1}}
\put(4,0){\circle*{0.1}}
\put(4,1){\circle*{0.1}}
\put(4,2){\circle*{0.1}}
\put(4,3){\circle*{0.1}}
\put(4,4){\circle*{0.1}}
\put(4,5){\circle*{0.1}}
\put(4,6){\circle*{0.1}}
\put(5,0){\circle*{0.1}}
\put(5,1){\circle*{0.1}}
\put(5,2){\circle*{0.1}}
\put(5,3){\circle*{0.1}}
\put(5,4){\circle*{0.1}}
\put(5,5){\circle*{0.1}}
\put(5,6){\circle*{0.1}}
\put(6,0){\circle*{0.1}}
\put(6,1){\circle*{0.1}}
\put(6,2){\circle*{0.1}}
\put(6,3){\circle*{0.1}}
\put(6,4){\circle*{0.1}}
\put(6,5){\circle*{0.1}}
\put(6,6){\circle*{0.1}}

\put(8,0){\circle*{0.1}}
\put(8,1){\circle*{0.1}}
\put(8,2){\circle*{0.1}}
\put(8,3){\circle*{0.1}}
\put(8,4){\circle*{0.1}}
\put(8,5){\circle*{0.1}}
\put(8,6){\circle*{0.1}}
\put(9,0){\circle*{0.1}}
\put(9,1){\circle*{0.1}}
\put(9,2){\circle*{0.1}}
\put(9,3){\circle*{0.1}}
\put(9,4){\circle*{0.1}}
\put(9,5){\circle*{0.1}}
\put(9,6){\circle*{0.1}}
\put(10,0){\circle*{0.1}}
\put(10,1){\circle*{0.1}}
\put(10,2){\circle*{0.1}}
\put(10,3){\circle*{0.1}}
\put(10,4){\circle*{0.1}}
\put(10,5){\circle*{0.1}}
\put(10,6){\circle*{0.1}}
\put(11,0){\circle*{0.1}}
\put(11,1){\circle*{0.1}}
\put(11,2){\circle*{0.1}}
\put(11,3){\circle*{0.1}}
\put(11,4){\circle*{0.1}}
\put(11,5){\circle*{0.1}}
\put(11,6){\circle*{0.1}}
\put(12,0){\circle*{0.1}}
\put(12,1){\circle*{0.1}}
\put(12,2){\circle*{0.1}}
\put(12,3){\circle*{0.1}}
\put(12,4){\circle*{0.1}}
\put(12,5){\circle*{0.1}}
\put(12,6){\circle*{0.1}}
\put(13,0){\circle*{0.1}}
\put(13,1){\circle*{0.1}}
\put(13,2){\circle*{0.1}}
\put(13,3){\circle*{0.1}}
\put(13,4){\circle*{0.1}}
\put(13,5){\circle*{0.1}}
\put(13,6){\circle*{0.1}}
\put(14,0){\circle*{0.1}}
\put(14,1){\circle*{0.1}}
\put(14,2){\circle*{0.1}}
\put(14,3){\circle*{0.1}}
\put(14,4){\circle*{0.1}}
\put(14,5){\circle*{0.1}}
\put(14,6){\circle*{0.1}}

\put(0.5,0.5){$[1/3]$}
\put(8.5,0.5){$[1/3]$}
\put(9.5,0.5){$[1/2]$}
\put(8.5,1.5){$[1/2]$}
\put(8.5,2.5){$[2/3]$}
\put(9.5,1.5){$[2/3]$}
\put(8.5,3.5){$[5/6]$}
\put(9.5,2.5){$[5/6]$}
\put(10.5,1.5){$[5/6]$}
\put(10.5,0.5){$[7/10]$}
\put(11.6,0.5){$[9/10]$}


\put(3.3,3.2){$P$}
\put(11.3,3.2){$P$}

\put(1.5,7){Monodromy at $0$}
\put(9.5,7){Monodromy at $\infty$}

\linethickness{0.075mm}
\put(0,5){\line(0,1){1}}
\put(5,0){\line(1,0){1}}
\put(1,2){\line(-1,3){1}}
\put(1,2){\line(1,-1){1}}
\put(2,1){\line(2,-1){2}}
\put(0,0){\line(2,1){2}}
\put(0,0){\line(1,2){1}}

\put(8,5){\line(0,1){1}}
\put(13,0){\line(1,0){1}}
\put(8,0){\line(1,5){1}}
\put(8,0){\line(2,1){4}}


\linethickness{0.2mm}
\put(0,0){\line(0,1){5}}
\put(0,0){\line(1,0){5}}
\put(5,0){\line(-1,2){1}}
\put(4,2){\line(-1,1){3}}
\put(0,5){\line(1,0){1}}

\put(8,0){\line(0,1){5}}
\put(8,0){\line(1,0){5}}
\put(13,0){\line(-1,2){1}}
\put(12,2){\line(-1,1){3}}
\put(8,5){\line(1,0){1}}


\end{picture}

\end{example}

\subsection{Jordan block structure of
monodromy at infinity}

We continue with the notation of Example~\ref{e:infty} above, and let  $f \in \C[x_1, \ldots, x_n]$ be convenient and sch\"on at infinity, and let $X = \{ ft = \lambda \} \subseteq \K^n$ for a general choice of $\lambda \in \C^*$.
Recall that
for every face $Q$ of the Newton polytope $\NP(f)$, we let
$\Delta_Q$ denote the convex hull of $Q$ and the origin. Recall that $P = \Delta_{\NP(f)}$, and $P_\infty$ denotes the union of the faces at infinity i.e. the faces of $P$ not containing the origin.
Recall that
  $\nu_\infty$ is the piecewise $\Q$-affine function on $P$ with value $0$ at the origin and value identically $1$ on the faces at infinity of $P$, and
the cells of $\cS(\nu_\infty)$ are given by the union of  $\{ Q \mid Q \subseteq P_\infty \}$  and $\{ \Delta_Q \mid Q \subseteq P_\infty \}$.
Recall that
$\Delta^\infty$ is the simplex $P \cap \{ x_1 + \cdots + x_n = \epsilon \}$ for
fixed $\epsilon > 0$ sufficiently small, and  $\cS_\infty$ denotes the regular, rational polyhedral subdivision of  $\Delta^\infty$ obtained by intersecting $\cS(\nu_\infty)$ with $\Delta^\infty$.

Following \cite{MTMonodromy}, one may use results of Broughton and Sabbah to read off the Jordan normal form of the action of monodromy on $H^{n - 1}_c (X_{\gen})$
from the equivariant limit Hodge-Deligne polynomial. We outline the argument below.
By Poincar\'{e} duality, we may work with usual cohomology $H^{n - 1} (X_{\gen})$ rather than cohomology with compact supports. We write $H^{n - 1} (X_{\gen}) = H^{n - 1} (X_{\gen})_{= 1} \oplus H^{n - 1} (X_{\gen})_{\ne 1}$, where $H^{n - 1} (X_{\gen})_{= 1}$ denotes the $1$-eigenspace of $H^{n - 1} (X_{\gen})$, and $H^{n - 1} (X_{\gen})_{\ne 1}$ denotes the sum of the $\alpha$-eigenspaces, for $\alpha \ne 1$.
The induced monodromy weight filtration $M_\bullet$ on $H^{n - 1} (X_{\gen})$
may be written
as $M_\bullet = (M_{=1})_\bullet \oplus (M_{\ne 1})_\bullet$, where $(M_{=1})_\bullet$ and  $(M_{\ne 1})_\bullet$ are the induced filtrations on $H^{n - 1} (X_{\gen})_{= 1}$ and $H^{n - 1} (X_{\gen})_{\ne 1}$
respectively. Let $N = \log T_u$ denote the nilpotent operator where $T_u$ is the unipotent part of monodromy acting on $H^{n - 1} (X_{\gen})$, and consider the induced action of $N$  on
$H^{n - 1} (X_{\gen})_{= 1}$ and $H^{n - 1} (X_{\gen})_{\ne 1}$.
As explained in detail in \cite{MTMonodromy}, using results of Broughton \cite{BroMilnor},
we have the following special case of a deep result of Sabbah. Recall the definition of the weight filtration of a nilpotent operator from Definition~\ref{d:weight}.

\begin{theorem} \cite[Theorem~13.1]{SabHyper} \cite[Proposition~A.1]{MTMonodromy} \label{t:sabbah}
Let  $f \in \C[x_1, \ldots, x_n]$ be convenient and sch\"on at infinity, and let $X = \{ ft = \lambda \} \subseteq \K^n$.
Then $(M_{=1})_\bullet$ (respectively   $(M_{\ne 1})_\bullet$) is equal to the $N$-weight filtration centered at $n$ (respectively $n - 1$).
\end{theorem}

Using Poincar\'{e} duality together with the symmetries of
Remark~\ref{r:termsym},  the corollary below follows immediately from Example~\ref{e:infty}. 

\begin{corollary}\label{c:formula}
 Let  $f \in \C[x_1, \ldots, x_n]$ be convenient and sch\"on at infinity, and let $X = \{ ft = \lambda \} \subseteq \K^n$.
 Then we have following formulas for  the equivariant limit mixed Hodge numbers of $H^{n - 1} (X_{\gen})$:
\begin{equation*}\label{e:nonzero}
uv \sum_{p,q} \sum_{ \substack{ \alpha \in \Q/\Z \\ \alpha \ne 1} } h^{p,q}( H^{n - 1} (X_{\infty}))_{\alpha} \alpha u^p v^q  = \sum_{Q \subseteq P_\infty}  v^{\dim \Delta_Q + 1}\lc(\Delta_Q , \nu_\infty|_{\Delta_Q }; uv^{-1})   l_{\Delta^\infty}(\cS_\infty,Q_\infty;uv),
\end{equation*}
\begin{equation*}\label{e:zero}
\sum_{p,q} h^{p,q}( H^{n - 1} (X_{\infty}))_{1}  u^p v^q  =  \sum_{Q \subseteq P_\infty}
v^{\dim Q + 1} \lc(Q;uv^{-1})  l_{\Delta^\infty}(\cS_\infty,Q_\infty;uv),
\end{equation*}
where the sum runs over all (possibly empty) faces at infinity.
\end{corollary}

Specializing the  formulas in Corollary~\ref{c:formula} by setting $u = v$ yields:
\[
u^2 \sum_{p,q} \sum_{ \substack{ \alpha \in \Q/\Z \\ \alpha \ne 1} } h^{p,q}( H^{n - 1} (X_{\infty}))_{\alpha} \alpha u^{p + q}  = \sum_{Q \subseteq P_\infty}  u^{\dim \Delta_Q + 1}\lc(\Delta_Q , \nu_\infty|_{\Delta_Q }; 1)   l_{\Delta^\infty}(\cS_\infty,Q_\infty;u^2),
\]
\[
\sum_{p,q} h^{p,q}( H^{n - 1} (X_{\infty}))_{1}  u^{p + q}  =  \sum_{Q \subseteq P_\infty}
u^{\dim Q + 1} \lc(Q;1) l_{\Delta^\infty}(\cS_\infty,Q_\infty;u^2).
\]
The following result now follows directly from Theorem~\ref{t:sabbah}. For every face at infinity $Q$ of $P$, it follows from Remark~\ref{r:non-negative} that we may write
\[ l_{\Delta^\infty}(\cS_\infty,Q_\infty;t) = \sum_{i = 0}^{\lfloor \frac{n - 1 - \dim Q }{2}\rfloor} \tilde{l}_{Q,i} t^i(1 + t + \cdots + t^{n - 1 - \dim Q - 2i}), \]
for some non-negative integers   $\{ \tilde{l}_{Q,i} \mid 0 \le i \le \lfloor \frac{n - 1 - \dim Q }{2}\rfloor \}$. We define
\[
\tilde{l}_{\Delta^\infty}(\cS_\infty,Q_\infty;t) 
:= \sum_{i = 0}^{\lfloor \frac{n - 1 - \dim Q }{2}\rfloor} \tilde{l}_{Q,i} t^i.
\]

\begin{corollary}\label{c:jordaninfty}
Let  $f \in \C[x_1, \ldots, x_n]$ be convenient and sch\"on at infinity, and let $J^\infty_{k,\alpha}$ be the number of Jordan blocks of size $k$ with eigenvalue $\alpha$ for the action of monodromy at infinity on
 $H^{n - 1} (f^{-1}(t))$ for $t$ fixed and sufficiently large. Then
 \[
\sum_{ \substack{ \alpha \in \Q/\Z \\ \alpha \ne 1} } \sum_k J^\infty_{n - k,\alpha} \alpha u^{k + 2} =  \sum_{Q \subseteq P_\infty} u^{\dim \Delta_Q + 1}\lc(\Delta_Q, \nu_\infty|_{\Delta_Q}; 1)
\tilde{l}_{\Delta^\infty}(\cS_\infty,Q_\infty;u^2), 
\]
\[
\sum_k J^\infty_{n - 1 - k,1} u^{k + 2} =   \sum_{Q \subseteq P_\infty} u^{\dim Q + 1}\lc(Q; 1)
\tilde{l}_{\Delta^\infty}(\cS_\infty,Q_\infty;u^2), 
\]
where the sum runs over all (possibly empty) faces at infinity.
\end{corollary}

\begin{example}\label{e:firstterms}
Let $\partial \R^n_{\ge 0}$ denote the intersection of $\R^n_{\ge 0}$ with the union of the coordinate hyperplances.
Using either Example~\ref{e:relative} and Example~\ref{e:weightedlocal}, or Example~\ref{e:smallterms} and Example~\ref{e:G}, one computes the number of Jordan blocks with  the largest and second largest possible size:
\[
\sum_{ \substack{ \alpha \in \Q/\Z \\ \alpha \ne 1} } J^\infty_{n,\alpha} \alpha = h^*_{0,0,n - 1}(P,\nu_\infty) =  \sum_{\substack{Q \subseteq P_\infty, Q \nsubseteq \partial \R^n_{\ge 0}
\\ \dim Q = 0}}  \sum_{v \in \Int(\Delta_Q) \cap \Z^n} w(v),
\]
\[
\sum_{ \substack{ \alpha \in \Q/\Z \\ \alpha \ne 1} } J^\infty_{n - 1,\alpha} \alpha = h^*_{0,1,n - 1}(P,\nu_\infty) + h^*_{1,0,n - 1}(P,\nu_\infty) =  \sum_{\substack{Q \subseteq P_\infty, Q \nsubseteq \partial \R^n_{\ge 0}
\\ \dim Q = 1}}  \sum_{v \in \Int(\Delta_Q) \cap \Z^n} w(v) + \bar{w(v)},
\]
\[
J^\infty_{n - 1,1} = \sum_{ \substack{ Q \subseteq P_\infty, Q \nsubseteq \partial \R^n_{\ge 0}\\ \dim Q \le 1}  }   \#( \Int(Q) \cap \Z^n),
\]
\[
J^\infty_{n - 2,1} = 2 \cdot  \sum_{ \substack{ Q \subseteq P_\infty, Q \nsubseteq \partial \R^n_{\ge 0}\\ \dim Q = 2}  }   \#( \Int(Q) \cap \Z^n).
\]
This reproves the main results of \cite{MTMonodromy}. Namely,  the first two equations are equivalent to \cite[Theorem~1.1]{MTMonodromy}, and the second two equations are equivalent to  \cite[Theorem~1.2]{MTMonodromy}.
\end{example}

\begin{example}\label{e:infty2}
As in Example~\ref{e:dim2}, consider the case when $n = 2$. Then $J^\infty_{1,1} = \#(\partial P \cap \Z^{2}_{> 0})$, and
\[
\sum_{ \substack{ \alpha \in \Q/\Z \\ \alpha \ne 1} } J^\infty_{2,\alpha} \alpha =  \sum_{\substack{Q \subseteq P_\infty, Q \nsubseteq \partial \R^2_{\ge 0}
\\ \dim Q = 0}}  \sum_{v \in \Int(\Delta_Q) \cap \Z^2} w(v),
\]
\[
\sum_{ \substack{ \alpha \in \Q/\Z \\ \alpha \ne 1} } J^\infty_{1,\alpha} \alpha =  \sum_{\substack{Q \subseteq P_\infty
\\ \dim Q = 1}}  \sum_{v \in \Int(\Delta_Q) \cap \Z^2} w(v) + \bar{w(v)}.
\]
\end{example}

\begin{example}\label{e:simpleagain2}
Following on with Example~\ref{e:simpleagain}, if  $f  = \sum_{i  = 1}^n 
x_i^{m_i} \in  \C[x_1,\ldots,x_n]$, for some  
$m_i \in \Z_{>0}$, then the action of monodromy is semi-simple, and all Jordan blocks have size $1$. With the notation of Example~\ref{e:simpleagain},
\[
\sum_{ \alpha \in \Q/\Z } J^\infty_{1,\alpha} \alpha  = \sum_{(i_1, \ldots i_n) \in [1,m_1 - 1] \times \cdots \times [1,m_n - 1]} [ \nu_\infty(i_1,\ldots,i_n)] \in \Z[\Q/\Z].
\]
\end{example}

\begin{example}\label{e:running}
As in Example~\ref{e:double}, let $f = a_0 x^4 + a_1 x^5 + a_2 x^4y^2 + a_3 xy^5 + a_4 y^5 + a_5 xy^2 + a_6 x^2y \in \C[x,y]$ for a general choice of  $a_0, \ldots, a_6 \in \C^*$.
The action of monodromy on  $H^{1} (X_{\gen})$ has 22 Jordan blocks of size $1$, $4$ with eigenvalue $1$, $3$ with eigenvalues $\exp(2\pi \sqrt{-1}/6)$,
 $\exp(2\pi \sqrt{-1}/3)$, $\exp(4\pi \sqrt{-1}/3)$ and  $\exp(10\pi \sqrt{-1}/6)$ respectively, $2$ with eigenvalue $-1$ and $1$ with eigenvalues
  $\exp(\pi \sqrt{-1}/5)$,
 $\exp(3\pi \sqrt{-1}/5)$, $\exp(7\pi \sqrt{-1}/5)$ and  $\exp(9\pi \sqrt{-1}/5)$ respectively, as well as a single Jordan block of size $2$ with eigenvalue $-1$.
\end{example}

\subsection{
Monodromy of Milnor fibers}\label{s:Milnor}

Finally, we present some local versions of the formulas above (see Example~\ref{e:0}).

Consider a polynomial  $f \in \C[x_1,\ldots, x_n]$  and consider the polynomial map $f: \C^n \rightarrow \C$. Assume
that $f^{-1}(0) \subseteq \C^n$ has an isolated singularity at the origin.
Restricting $f : \C^n \rightarrow \C$
to a small ball about the origin in $\C^n$, and replacing $\C$ with a sufficiently small punctured disc about the origin, gives a locally trivial fibration called the \define{Milnor fibration}. Fix a fiber $F_0$, called the
\define{Milnor fiber}.
A fundamental result of Milnor  asserts that $F_0$ has the homotopy type of a wedge of $(n - 1)$-spheres \cite{Milnor}. In particular, $H^m(F_0) = 0$ unless $m = 0, n - 1$.  Here the monodromy action on
 $H^0(F_0) = \C$ is trivial. In \cite{SteMixed}, Steenbrink introduced a
mixed Hodge structure $(F^\bullet, M_\bullet)$ on  $H^{n - 1}(F_0)$.   The weight filtration has the following description in terms of the induced monodromy map $T = T_sT_u$ on
$H^{n - 1}(F_0)$ (cf. Theorem~\ref{t:sabbah}). We write
$H^{n - 1} (F_0) = H^{n - 1} (F_0)_{= 1} \oplus H^{n - 1} (F_0)_{\ne 1}$, where $H^{n - 1} (F_0)_{= 1}$ denotes the $1$-eigenspace of $H^{n - 1} (F_0)$, and $H^{n - 1} (F_0)_{\ne 1}$ denotes the sum of the $\alpha$-eigenspaces, for $\alpha \ne 1$.
Then 
 $M_\bullet = (M_{=1})_\bullet \oplus (M_{\ne 1})_\bullet$, where $(M_{=1})_\bullet$ and  $(M_{\ne 1})_\bullet$ are the induced filtrations on $H^{n - 1} (F_0)_{= 1}$ and $H^{n - 1} (F_0)_{\ne 1}$
respectively.
Let $N = \log T_u$ denote the nilpotent operator acting on $H^{n - 1} (F_0)$, and consider the induced action of $N$  on
$H^{n - 1} (F_0)_{= 1}$ and $H^{n - 1} (F_0)_{\ne 1}$.
Then $(M_{=1})_\bullet$ (respectively   $(M_{\ne 1})_\bullet$) is equal to the $N$-weight filtration centered at $n$ (respectively $n - 1$) (see Definition~\ref{d:weight}).
Denef and Loeser introduced the \define{motivic Milnor fiber} \cite{DLMotivic}, which specializes under the equivariant Hodge-Deligne map $E_{\hat{\mu}}$ to the equivariant Hodge-Deligne polynomial
$E(F_0, \hat{\mu}; u,v) \in \Z[\Q/\Z][u,v]$ associated to $F_0$ with the above mixed Hodge structure. Explicitly,
\[
E(F_0, \hat{\mu}; u,v) := \sum_{p,q} \sum_{\alpha \in \Q/\Z} \sum_m (-1)^m h^{p,q}(H^m(F_0))_\alpha \alpha u^p v^q.
\]

Recall that $\Gamma_+(f) = \NP(f) + \R_{\ge 0}^n$ denotes the \define{Newton polyhedron} of $f$, and $\Gamma_f$ denotes the union of the bounded faces of $\Gamma_+(f)$.
Recall that for every bounded face $Q$ of $\Gamma_+(f)$, $\Delta_Q$ denotes the convex hull of $Q$ and the origin. We assume that $f$ is convenient i.e.
$\Gamma_+(f)$ has non-zero intersection with each  ray through a coordinate vector. We assume that $f$ is \define{sch\"on at $0$},
meaning that  $\{ f|_Q =  0 \} \subseteq (\C^*)^n$ defines a smooth hypersurface whenever $Q$ is a bounded face of $\Gamma_+(f)$. Note that this is a weaker condition than the condition that $f$ is sch\"on in Example~\ref{e:0}.
Let $P_+$ be the union of $\{ \Delta_Q \mid Q \subseteq \Gamma_f \}$, and
let $\nu_0$ be the piecewise $\Q$-affine function on $P_+$ with value $1$ at the origin and value $0$ on $\Gamma_f$. Then the lattice polyhedral decomposition $\cS(\nu_0)$ of $P_+$ has cells
$\{ \Delta_Q \mid Q \subseteq \Gamma_f \}$ and $\{  Q \subseteq \Gamma_f \}$. Although we will not need this, we note that even though $P_+$ is not convex, one can apply all the combinatorial constructions and results of
Section~\ref{s:weighted}. With the notation of Example~\ref{e:0}, the following formula for the motivic Milnor fiber was given in \cite[Theorem~4.3]{MTMilnor} (cf. \eqref{e:mnf0}):
\[
\sum_{\emptyset \ne Q \subseteq \Gamma_f } (1 - \L)^{\dim \sigma(\Delta_Q) - \dim \Delta_Q} \big [ [V_{\Delta_Q}^\circ \circlearrowleft \hat{\mu}]  +  
 [V_{Q}^\circ](1 - \L) \big],
\]
where $\sigma(\Delta_Q)$ is the smallest face of $P_+$ containing $\Delta_Q$.
Specializing the above expression via the equivariant Hodge-Deligne map and applying Theorem~\ref{t:DKformula} yields the following formula for the equivariant Hodge-Deligne polynomial:
\[
uvE(F_0, \hat{\mu}; u,v)  = \sum_{\emptyset \ne  Q \subseteq \Gamma_f} (-1)^{\dim Q} (1 - uv)^{\dim \sigma(\Delta_Q) - \dim \Delta_Q}
\big[ h^*(\Delta_Q , \nu_0|_{\Delta_Q }; u,v) + (uv - 1) h^*(Q;u,v)  \big],
\]
where $h^*(Q;u,v)$ is the weighted limit mixed $h^*$-polynomial with respect to the convex graph that is identically zero.
This gives the following expression for $\sum_{p,q} \sum_{\alpha \in \Q/\Z}  h^{p,q}(H^{n - 1}(F_0))_\alpha \alpha u^p v^q $: 
\[
\sum_{Q \subseteq \Gamma_f} (-1)^{n - 1- \dim Q} (1 - uv)^{\dim \sigma(\Delta_Q) - \dim \Delta_Q}
\big[ h^*(\Delta_Q , \nu_0|_{\Delta_Q }; u,v) + (uv - 1) h^*(Q;u,v)  \big].
\]

\begin{remark}\label{r:var}
The formula below should be compared to \eqref{e:0eigen}.
Note that $E(F_0, \hat{\mu}; 1,1) =   \sum_{\alpha \in \Q/\Z} \sum_m (-1)^m \dim H^m(F_0)_\alpha \alpha$ determines the eigenvalues (with multiplicity)  of the action of monodromy.
Specializing the above formula at $u = v = 1$, and applying Example~\ref{e:volume}, gives the following formula:
\[
 \sum_{\alpha \in \Q/\Z} \dim H^{n - 1}(F_0)_\alpha \alpha = \sum_{  \substack{ Q \subseteq \Gamma_f \\ \dim \sigma(\Delta_Q) = \dim \Delta_Q}} (-1)^{n - 1 - \dim Q} \Vol(Q) \sum_{i = 0}^{m(Q) - 1} [i/m(Q)],
\]
where $\Vol(Q)$ is the normalized volume of $Q$ and $m(Q)$ is the lattice distance of $Q$ from the origin i.e. $m(Q)$ is
the minimal positive integer such that  $m(Q) \nu_0|_{\Delta_Q}$ is an affine 
function with respect to the lattice given by intersecting $M$ with the affine span of $\Delta_Q$. When $Q$ is empty,
$\Vol(Q) = m(Q) = 1$ and $\dim Q = -1$.
This is equivalent to  a famous  formula of Varchenko \cite{VarZeta} for the zeta function of monodromy of $F_0$. As in Remark~\ref{r:forget}, we obtain a  formula for $\dim H^{n - 1}(F_0)$ originally due to
Kouchnirenko \cite{Kou}.
\end{remark}

As in previous sections, expanding the definitions above, simplifying and using Remark~\ref{non-zero}, yields the following `non-negative' formulas for the equivariant mixed Hodge numbers of  $F_0$.
Let $\Delta^0$ be the simplex $P_+ \cap \{ x_1 + \cdots + x_n = \epsilon \}$ for
fixed $\epsilon > 0$ sufficiently small. Let $\cS_0$ denote the regular, rational polyhedral subdivision obtained by intersecting $\cS(\nu_0)$ with $\Delta^0$. That is, the cells of $\cS_0$ are
$\{ Q_0 := \Delta_Q \cap \Delta^0 \mid Q \subseteq \Gamma_f \}$.

\begin{theorem}\label{t:milnorformula}
Let  $f \in \C[x_1,\ldots, x_n]$  be a complex polynomial such that $f^{-1}(0)$ admits an isolated singularity at the origin. Assume further that $f$ is convenient and sch\"on at $0$. With the notation above, we have the following formulas for the
equivariant mixed Hodge numbers of the cohomology of the associated Milnor fiber $F_0$:
\begin{equation*}
uv \sum_{p,q} \sum_{ \substack{ \alpha \in \Q/\Z \\ \alpha \ne 1} } h^{p,q}( H^{n - 1} (F_0))_{\alpha} \alpha u^p v^q  = \sum_{Q \subseteq \Gamma_f}  v^{\dim \Delta_Q + 1}\lc(\Delta_Q , \nu_0|_{\Delta_Q }; uv^{-1})  l_{\Delta^0}(\cS_0,Q_0;uv), 
\end{equation*}
\begin{equation*}
\sum_{p,q} h^{p,q}( H^{n - 1} (F_0))_{1}  u^p v^q  =  \sum_{Q \subseteq \Gamma_f}
v^{\dim Q + 1} \lc(Q;uv^{-1}) l_{\Delta^0}(\cS_0,Q_0;uv), 
\end{equation*}
where the sum runs over all (possibly empty) bounded faces of the Newton polyhedron of $f$.
\end{theorem}

We note that an algorithm to compute the equivariant mixed Hodge numbers of the cohomology of the associated Milnor fiber $F_0$, as well as formulas in special cases,  were given by 
Matsui and
Takeuchi in  \cite{MTMilnor}, extending work of Danilov  \cite{DanNewton} and Tanab\'e \cite{TanCombinatorial}.

\begin{remark}
There is a striking symmetry between the formulas for monodromy at infinity and monodromy of the Milnor fiber in  Corollary~\ref{c:formula} and Theorem~\ref{t:milnorformula} respectively. The existence of
such a symmetry was first observed by Matsui and Takeuchi in \cite{MTMilnor}.
\end{remark}

The following corollary is immediate from the description of the weight filtration on the cohomology of $F_0$ and Definition~\ref{d:weight}, and should be compared to Corollary~\ref{c:jordaninfty}.
 For every bounded face $Q$ of $\Gamma_f$, it follows from Remark~\ref{r:non-negative} that we may write
\[ l_{\Delta^0}(\cS_0,Q_0;t) = \sum_{i = 0}^{\lfloor \frac{n - 1 - \dim Q }{2}\rfloor} \tilde{l}_{Q,i} t^i(1 + t + \cdots + t^{n - 1 - \dim Q - 2i}), \]
for some non-negative integers   $\{ \tilde{l}_{Q,i} \mid 0 \le i \le \lfloor \frac{n - 1 - \dim Q }{2}\rfloor \}$. We define
\[
\tilde{l}_{\Delta^0}(\cS_0,Q_0;t)  
:= \sum_{i = 0}^{\lfloor \frac{n - 1 - \dim Q }{2}\rfloor} \tilde{l}_{Q,i} t^i.
\]

\begin{corollary}\label{c:jordanMilnor}
Let  $f \in \C[x_1,\ldots, x_n]$  be a complex polynomial such that $f^{-1}(0)$ admits an isolated singularity at the origin. Assume further that $f$ is convenient and sch\"on at $0$.
Let $J^0_{k,\alpha}$ be the number of Jordan blocks of size $k$ with eigenvalue $\alpha$ for the action of monodromy on the cohomology of the Milnor fiber
 $H^{n - 1} (F_0)$. Then
 \[
\sum_{ \substack{ \alpha \in \Q/\Z \\ \alpha \ne 1} } \sum_k J^0_{n - k,\alpha} \alpha u^{k + 2} =  \sum_{Q \subseteq \Gamma_f} u^{\dim \Delta_Q + 1}\lc(\Delta_Q, \nu_0|_{\Delta_Q}; 1)
 \tilde{l}_{\Delta^0}(\cS_0,Q_0;u^2), 
\]
\[
\sum_k J^0_{n - 1 - k,1} u^{k + 2} =   \sum_{Q \subseteq \Gamma_f} u^{\dim Q + 1}\lc(Q; 1)
 \tilde{l}_{\Delta^0}(\cS_0,Q_0;u^2), 
\]
where the sum runs over all (possibly empty) bounded faces of the Newton polyhedron of $f$.
\end{corollary}

\begin{example}\label{e:firsttermsmilnor}
We have the following analogue of Example~\ref{e:firstterms}.
Let $\partial \R^n_{\ge 0}$ denote the intersection of $\R^n_{\ge 0}$ with the union of the coordinate hyperplances.
Then we have the following formulas for the number of Jordan blocks with  the largest and second largest possible size:
\[
\sum_{ \substack{ \alpha \in \Q/\Z \\ \alpha \ne 1} } J^0_{n,\alpha} \alpha = \sum_{\substack{Q \subseteq \Gamma_f, Q \nsubseteq \partial \R^n_{\ge 0}
\\ \dim Q = 0}}  \sum_{v \in \Int(\Delta_Q) \cap \Z^n} w(v),
\]
\[
\sum_{ \substack{ \alpha \in \Q/\Z \\ \alpha \ne 1} } J^0_{n - 1,\alpha} \alpha =  \sum_{\substack{Q \subseteq \Gamma_f, Q \nsubseteq \partial \R^n_{\ge 0}
\\ \dim Q = 1}}  \sum_{v \in \Int(\Delta_Q) \cap \Z^n} w(v) + \bar{w(v)},
\]
\[
J^0_{n - 1,1} = \sum_{ \substack{ Q \subseteq \Gamma_f, Q \nsubseteq \partial \R^n_{\ge 0}\\ \dim Q \le 1}  }   \#( \Int(Q) \cap \Z^n),
\]
\[
J^0_{n - 2,1} = 2 \cdot  \sum_{ \substack{ Q \subseteq  \Gamma_f, Q \nsubseteq \partial \R^n_{\ge 0}\\ \dim Q = 2}  }   \#( \Int(Q) \cap \Z^n).
\]
This reproves the main results of \cite{MTMilnor}. Namely,  the first two equations are equivalent to \cite[Theorem~1.1]{MTMilnor}, and the second two equations are equivalent to  \cite[Theorem~1.2]{MTMilnor}. Also, the third equation is originally due to van Doorn and Steenbrink \cite{DSSupplement}.
\end{example}

\begin{example}
We have the following analogue of Example~\ref{e:infty2}. Consider the case when $n = 2$.
 Then $J^0_{1,1} = \#(\Gamma_f \cap \Z^{2}_{> 0})$, and
\[
\sum_{ \substack{ \alpha \in \Q/\Z \\ \alpha \ne 1} } J^0_{2,\alpha} \alpha =  \sum_{\substack{Q \subseteq \Gamma_f, Q \nsubseteq \partial \R^2_{\ge 0}
\\ \dim Q = 0}}  \sum_{v \in \Int(\Delta_Q) \cap \Z^2} w(v),
\]
\[
\sum_{ \substack{ \alpha \in \Q/\Z \\ \alpha \ne 1} } J^0_{1,\alpha} \alpha =  \sum_{\substack{Q \subseteq \Gamma_f
\\ \dim Q = 1}}  \sum_{v \in \Int(\Delta_Q) \cap \Z^2} w(v) + \bar{w(v)}.
\]
\end{example}

\begin{example}
We have the following analogue of Example~\ref{e:simpleagain2}.
Following on with Example~\ref{e:simpleagain}, if  $f  = \sum_{i  = 1}^n 
x_i^{m_i} \in  \C[x_1,\ldots,x_n]$, for some  
$m_i \in \Z_{>0}$, then the action of monodromy is semi-simple, and all Jordan blocks have size $1$. With the notation of Example~\ref{e:simpleagain},
\[
\sum_{ \alpha \in \Q/\Z } J^0_{1,\alpha} \alpha  = \sum_{(i_1, \ldots i_n) \in [1,m_1 - 1] \times \cdots \times [1,m_n - 1]} [ \nu_\infty(i_1,\ldots,i_n)] \in \Z[\Q/\Z].
\]
\end{example}

\begin{example}
The following should be compared to  Example~\ref{e:running}.
As in Example~\ref{e:double}, let $f = a_0 x^4 + a_1 x^5 + a_2 x^4y^2 + a_3 xy^5 + a_4 y^5 + a_5 xy^2 + a_6 x^2y \in \C[x,y]$ for a general choice of  $a_0, \ldots, a_6 \in \C^*$.
The action of monodromy on  $H^{1} (F_0)$ is semi-simple and has $4$ Jordan blocks of size $1$, $2$ with eigenvalue $1$ and $1$ with eigenvalues
 $\exp(2\pi \sqrt{-1}/3)$ and $\exp(4\pi \sqrt{-1}/3)$ respectively.
\end{example}

\bibliographystyle{amsplain}
\def\cprime{$'$}
\providecommand{\bysame}{\leavevmode\hbox to3em{\hrulefill}\thinspace}
\providecommand{\MR}{\relax\ifhmode\unskip\space\fi MR }
\providecommand{\MRhref}[2]{%
  \href{http://www.ams.org/mathscinet-getitem?mr=#1}{#2}
}
\providecommand{\href}[2]{#2}

\end{document}